\newcommand{\simplestyle}[1]{#1}
\newcommand{\coapstyle}[1]{}

\coapstyle{
\documentclass[sn-mathphys-num]{sn-jnl}}
\simplestyle{
\documentclass{article}
}

\usepackage{placeins}
\usepackage{graphicx}\usepackage{multirow}\usepackage{amsmath,amssymb,amsfonts}\usepackage{amsthm}\usepackage{mathrsfs}\usepackage[title]{appendix}\usepackage[dvipsnames,table]{xcolor}\usepackage{textcomp}\usepackage{manyfoot}\usepackage{booktabs}\usepackage{algorithm}\usepackage{algorithmicx}\usepackage{algpseudocode}\usepackage{listings}\usepackage{color}

\usepackage{hyperref}
\simplestyle{
\usepackage{authblk}
}

\usepackage{lmodern}        \usepackage[utf8]{inputenc} \usepackage[T1]{fontenc}    \usepackage[english]{babel}
\usepackage{enumitem}
\usepackage{subcaption} 
\usepackage{diagbox} 
\usepackage{colortbl}
\usepackage[normalem]{ulem}
\usepackage[dvipsnames]{xcolor}
\usepackage[color=Goldenrod,textsize=small]{todonotes}

\makeatletter
\newcommand{\algmargin}{\the\ALG@thistlm}
\makeatother

\coapstyle{\newcommand{\change}[1]{{\textcolor{red}{#1}}}}
\simplestyle{\newcommand{\change}[1]{{\textcolor{black}{#1}}}}

\newcommand{\B}{\mathcal{B}}
\newcommand{\real}{{\mathbb{R}}}
\newcommand{\R}{\real}
\newcommand{\st}{\text{s.t.}\: }

\newcommand{\J}{{\mathcal{J}}}
\newcommand{\I}{{\mathcal{I}}}
\newcommand{\X}{{\mathcal{X}}}

\newcommand{\hbeta}[1]{h^{\beta}_#1}
\newcommand{\hbetak}[1]{h^{\beta_{k}}_#1}

\renewcommand{\S}{{\mathcal{S}}}

\DeclareMathOperator*{\argmin}{argmin}

\newcommand{\ourint}[2]{|(#1, #2)|}
\newcommand{\sourint}[2]{(#1, #2)}

\coapstyle{
\theoremstyle{thmstyleone}
}
\newtheorem{lemma}{Lemma}
\newtheorem{definition}{Definition}
\newtheorem{theorem}{Theorem}
\newtheorem{corollary}{Corollary}
\newtheorem{proposition}{Proposition}

\newtheorem{assumption}{Assumption}

\algblockdefx{ParDo}{EndParDo}{\textbf{do in parallel}}{\textbf{end parallel do}}

\algblockdefx{ParFor}{EndParFor}[1]{\textbf{for each} #1 \textbf{do in parallel}}
{\textbf{end parallel for}}

\algnewcommand{\parState}[1]{\State \parbox[t]{\dimexpr\linewidth-\algmargin}
{\strut #1}}

\begin{document}

\title{
    Parallel block coordinate descent methods with identification
    strategies
    \thanks{Partially supported by FAPESP grants 2023/08706-1, 2018/24293-0, 2014/14228-6, 2013/07375-0, and CNPq grants 312394/2023-3, 305010/2020-4.
    }
}

\coapstyle{
\author[1]{\fnm{Ronaldo} \sur{Lopes}}\email{ronaldolps3@gmail.com}
\equalcont{These authors contributed equally to this work.}
\author[2]{\fnm{Sandra A.} \sur{Santos}}\email{sandra@ime.unicamp.br}
\equalcont{These authors contributed equally to this work.}
\author*[2]{\fnm{Paulo J. S.} \sur{Silva}}\email{pjssilva@unicamp.br}
\equalcont{These authors contributed equally to this work.}

\affil*[1]{\orgdiv{Centro de Ci\^encias Exatas (CCE)}, \orgname{Universidade Estadual de Maring\'a (UEM)}, \orgaddress{\street{Av. Colombo, 5790 - Zona 7}, \city{Maring\'a}, \postcode{87020-900}, \state{Paran\'a}, \country{Brazil}}}

\affil*[2]{\orgdiv{Instituto de Matem\'atica, Estat\'{\i}stica e Computa\c{c}\~ao Cient\'{\i}fica (IMECC)}, \orgname{Universidade Estadual de Campinas (UNICAMP)}, \orgaddress{\street{Rua S\'ergio Buarque de Holanda, 651}, \city{Campinas}, \postcode{13083-859}, \state{S\~ao Paulo}, \country{Brazil}}}
}

\simplestyle{
\author[1]{Ronaldo Lopes}
\author[2]{Sandra A. Santos}
\author[2]{Paulo J. S. Silva}
\affil[1]{Universidade Estadual de Maring\'a (UEM),
Centro de Ci\^encias Exatas (CCE),
Av. Colombo, 5790 - Zona 7,
87020-900, Maring\'a, Paran\'a, Brazil
Email: {\tt ronaldolps3@gmail.com}}
\affil[2]{Universidade Estadual de Campinas (Unicamp),
Instituto de Matem\'atica, Estat\'{\i}stica e Computa\c{c}\~ao Cient\'{\i}fica (IMECC),
Rua Sergio Buarque de Holanda, 651,
13083-859, Campinas, S\~ao Paulo, Brazil
Emails: {\tt  sandra@ime.unicamp.br,
pjssilva@ime.unicamp.br}}
\renewcommand\Affilfont{\itshape\small}
}

\date{July 29th, 2025}

\coapstyle{
    \abstract{
        \change{
            This work presents a parallel variant of the algorithm introduced in
                [\emph{Acceleration of block coordinate descent methods with identification
                        strategies}, Comput. Optim. Appl. 72(3):609--640, 2019] to minimize the sum
            of a partially separable smooth convex function and a possibly non-smooth
            block-separable convex function under simple constraints. It achieves better
            efficiency by using a strategy to identify the nonzero coordinates that
            allows the computational effort to be focused on using a nonuniform probability
            distribution in the selection of the blocks. Parallelization is achieved
            by extending the theoretical results from Richt{\'a}rik and Tak{\'a}{\v{c}}
                [\emph{Parallel coordinate descent methods for big data optimization}, Math.
                    Prog. Ser. A 156:433--484, 2016]. We present convergence results and
            comparative numerical experiments on regularized regression problems
            using both synthetic and real data.
        }
    }
}

\coapstyle{
}

\maketitle

\simplestyle{
    \begin{abstract}
        \change{
            This work presents a parallel variant of the algorithm introduced in
                [\emph{Acceleration of block coordinate descent methods with identification
                        strategies}, Comput. Optim. Appl. 72(3):609--640, 2019] to minimize the sum
            of a partially separable smooth convex function and a possibly non-smooth
            block-separable convex function under simple constraints. It achieves better
            efficiency by using a strategy to identify the nonzero coordinates that
            allows the computational effort to be focused on using a nonuniform probability
            distribution in the selection of the blocks. Parallelization is achieved
            by extending the theoretical results from Richt{\'a}rik and Tak{\'a}{\v{c}}
                [\emph{Parallel coordinate descent methods for big data optimization}, Math.
                    Prog. Ser. A 156:433--484, 2016]. We present convergence results and
            comparative numerical experiments on regularized regression problems
            using both synthetic and real data.
        }

        \medskip

    \end{abstract}
}

\selectlanguage{english}

\section{Introduction}
\label{sec:int}

The contemporary need to address problems with huge data sets has renewed interest in simple first-order \change{optimization algorithms} such as coordinate descent methods, \change{particularly when} problem structure can be further explored~\cite{BT2013,N2012,rt14}. Another room for improvement is \change{the use of parallelism~\cite{CFKS2020,LW2015,NC2016,RT2016,TTR2018,TSV2023}, where decomposition strategies, communication, and distributed memory architectures are crucial~\cite{CLY2022,RT2016b,TRB2015,DFDM2019}.}

Previously\change{, \cite{LSS19} proposed an acceleration of block coordinate descent methods (BCDM\change{s}) using identification strategies. The main idea was to use a nonuniform probability distribution to select free blocks that are more likely to induce a larger reduction of the objective function at each iteration. The resulting method was called Active BCDM. This approach was combined with} an extra second-order step \change{in the subspace of the free \change{blocks}, using a \change{blocked Hessian, reducing} the running time required to solve $\ell_1$-regularized problems.} In \cite{NLS2022}, the authors develop an exhaustive study concerning algorithmic choices that influence the efficiency of BCDMs, \change{namely the block partitioning strategy, the block selection rule, and the block update rule.} In the current work, we investigate how \change{identification strategies can be used in parallel variants of BCDM.}

\change{A naive parallel BCDM implementation suffers from several challenges, notably the need to predefine subgroups of blocks to be updated in parallel. This would enable the precomputation of each subgroup's Lipschitz constants of the gradient. Without this precomputation, the descent directions and the Lipschitz constants, or line searches, would need to be computed for each new selection, which would limit computational performance.} 

\change{These issues were addressed by the theoretical results in~\cite{RT2016}. In this paper, the authors introduced the concept of partial separability, which relaxes the notion of function separability. Then, they developed a framework where the blocks can be updated in parallel using the original Lipschitz continuity constants of the gradient of each block. Hence, the blocks updated in parallel are no longer fixed in advance. The method also avoids extra function evaluations to enforce a descent criterion. Instead, it penalizes the curvature of the model used to compute the descent directions, leading to improved efficiency.}

\change{However, the results in~\cite{RT2016} assume that the blocks are selected using a uniform distribution, limiting the direct application of the identification strategy described in~\cite{LSS19}. The main theoretical contribution of the present text is to overcome this limitation and allow nonuniform distributions to be employed. This opens up the path to incorporate identification techniques in parallel block coordinate descent methods. We then present a convergence theory for our method, named Parallel Active BCDM (Algorithm~\ref{alg:ActivePCDM}) and compare its efficiency with the (uniform) parallel block coordinate descent method, \texttt{PCDM}. The numerical comparison uses randomly generated artificial tests and a library of real-world problems~\cite[Tables 2 and 3]{LSS19}. The numerical experiments unveil essential differences in the behavior of parallel block coordinate descent methods depending on the relation between dimensions of the data matrix when applied to regression problems with regularization (Lasso)~\cite{TB96}}.

\change{The remainder of this text is organized as follows:} Section~\ref{sec:background} contains the problem \change{statement} and \change{details} the main ingredients of Active BCDM, whose \change{(sequential) version is detailed} in Section~\ref{sec:BCDM}. \change{Then, its} parallel version \change{is introduced} in Section~\ref{sec:PCDM}, \change{along with the associated}  convergence results. \change{Finally,} computational experiments are presented and analyzed in Section~\ref{sec:numerical}. Our conclusions are stated in Section~\ref{sec:remarks}, \change{along with} lines for future research.
 \selectlanguage{english}

\section{Background}
\label{sec:background}

We \change{consider} the problem
\begin{align} \label{eq:mainprob}
  \min_{x \in \R^n} &\quad  f(x) + \psi(x) \\
  \st  &\quad  l\le x \le u, \notag
\end{align}
where $f:\R^n\rightarrow \R$ and $\psi:\R^n\rightarrow \R$ are both convex,
with $\psi$ possibly nonsmooth. In addition, $\psi$ is \change{assumed}
to have a block-separable structure given by
\begin{equation} \label{eq:psi}
  \psi(x) = \sum\limits_{i=1}^{m}\psi_i(x_{(i)}),
\end{equation}
where $x_{(i)}\in \R^{p_i}$ is a subset of coordinates of the vector $x \in
\R^n$ with $\sum_{i=1}^{m}p_i=n$. The vectors $l$ and $u$ have $n$
coordinates in $[-\infty, +\infty]$, with components $l_j < u_j$, $j=1,
\ldots, n$. Infinite values denote the absence of a bound.

Based on the block decomposition of the vector $x$ in $m$ coordinate subsets
described in $\eqref{eq:psi}$, we define a set of matrices $ U_i\in
\R^{n\times p_i}$ whose columns are canonical vectors of $\R^n$, and such that
\begin{equation*}
  x = \sum\limits_{i=1}^{m}U_ix_{(i)} \quad \text{and} \quad x_{(i)} = U_i^T x.
\end{equation*}

Let $B_i\in\mathbb{S}^{p_i}_{++}$, $i=1, \ldots, m$ be a set of positive
definite matrices \change{of} order $p_i$. We then define the respective 
primal and dual norms in $\R^{p_i}$ by
\begin{equation*}
  \|x_{(i)}\|_{(i)}:=\sqrt{x_{(i)}^{T}B_ix_{(i)}} \ \mbox{and} \
  \|x_{(i)}\|_{(i)}^{*}:=\sqrt{x_{(i)}^{T}B_i^{-1}x_{(i)}}, \ \forall i \in
  \{1,\dots,m\}.
\end{equation*}
Such matrices may be used to provide scaling information for the problem
variables, particularly if they are diagonal. They may also be used to provide
second-order information to the models.

The gradient of the smooth function $f$ is assumed to be Lipschitz continuous
by blocks, that is, for each $i = 1, \dots, m$, there exists a constant $L_i > 0$
such that 
\begin{equation} \label{eq:lipcondblk}
  \|\nabla _if(x+U_ih_i)-\nabla _if(x)\|_{(i)}^{*}\le L_i\|h_i\|_{(i)},\ h_i\in
  \R^{p_i},\ i=1,\dots, m, 
\end{equation}
for all $x$ such that $l\le x \le u$ and 
$\nabla _i f(x) = U_i^T\nabla f(x)$.

Let $B \in \R^{n \times n}$ be the block diagonal positive definite matrix
stated as 
\[
B := \text{diag}(L_1 B_1, \cdots, L_m B_m) \in \mathbb{R}^{n \times n},
\]
where, for  $i = 1, \ldots, m$, the scalars $L_i$ are as
in~\eqref{eq:lipcondblk} and the matrices $B_i$ are fixed above. We then
define the following primal and \change{dual} norms in $\mathbb{R}^n$:
\[
\|z\|_B = \sqrt{z^T B z} \quad \text{and} \quad 
\|z\|_B^* = \sqrt{z^T B^{-1} z}.
\]

For the vectors $x, y, l, u$, which can be indexed by coordinates or
blocks, we have adopted the convention that the subindex~$(i)$ refers to the
$i$th block, whereas $i$ refers to the $i$th component. The remaining vectors,
matrices, functions, and scalars used in the text, namely, $h_i$, $w_i$,
$p_i$, $s_i$, $B_i$, $U_i$, $\psi_i$, $\nabla_i f$, $L_i$, are only indexed by
blocks. Therefore, we simply use $i$ to refer to the $i$th block and avoid
overloading the notation. 

Throughout the text, $\|\cdot\|$ refers to either the Euclidean vector norm or
it's induced matrix norm. The cardinality of the set ${\cal I}$ is denoted by
$|{\cal I}|$. Given a positive integer $p$, we define the set of indices
$[p]:= \{1, \ldots, p \}$. \change{The feasible set of
problem~\eqref{eq:mainprob}} is denoted by
\begin{equation}\label{eq:box}
  \X:= \{ x \in \R^n \mid l \leq x \leq u\},
\end{equation}
and its objective function by $F(x):= f(x) + \psi(x)$.

 \selectlanguage{english}

\section{On the sequential Active BCDM}
\label{sec:BCDM}

Our objective is to parallelize the block coordinate descent method with
the identification of active variables (Active BCDM) \change{introduced}
in~\cite{LSS19}\change{, see Algorithm~\ref{alg:ActiveBCDM}}. \change{Let us}
start describing its essential components.

Given a vector $x \in \R^n$ and \change{a} block index $i$, the block descent
direction employed in the Active BCDM, $h_i(x)$, is the minimizer of the
subproblem
\begin{equation} \label{eq:subprob}
  \begin{aligned}
    \min_{h \in \R^{p_i}} &\quad \nabla_i f(x)^{T} h + \dfrac{L_i}{2}
    \| h\|_{(i)}^2 + \psi_i(x_{(i)}+h) - \psi_i(x_{(i)}),\\
    \st        &\quad l \le x+U_ih \le u.
  \end{aligned}
\end{equation}
Such descent directions were introduced in~\eqref{eq:subprob} and \change{they} are
supported theoretically by the fact that null descent directions for all blocks
\change{are} equivalent to the stationarity of the current iterate~\cite[Lemma 2]{LSS19}.

\begin{algorithm}[!ht]
  \caption{BCDM with Identification of Active Variables (Active BCDM)}
  \label{alg:ActiveBCDM}
  
  \begin{algorithmic}[1]
    
    \State Choose an initial point $l \le x^0 \le u$, \change{a block separation of $\psi$ function described as $x = \sum_{i=1}^{m}U_ix_{(i)}$,} an initial cycle
    size $c_0$, the parameters $\ell_{\max} \in \mathbb{N}$,
    $\varepsilon\in\R_{+}$ and two natural numbers
    $\delta_F, \delta_{DP}$. Initialize a vector $v\in\R^n$ with
    $2\varepsilon$ in all positions, the sets $\mathcal{I}=[m]$,
    $\mathcal{J}=\emptyset$, the cycle size $c_s=c_0$, and the counter
    $\ell = 0$.
    \State \change{Calculate the Lipschitz constants of the gradient by blocks satisfying~\eqref{eq:lipcondblk}.}
    \Repeat
    \For {$k = \ell+ 1, \ldots, \ell + c_s$}
    
    \parState{Choose a block $i$ that satisfies the probability
    distribution
    $$
    \mathbb{P}(i) = \begin{cases}
      \dfrac{\delta_{DP}}{\delta_{DP}|\mathcal{I}| + |\mathcal{J}|}, &
      \text{if}\ i \in \mathcal{I},\\[12pt]
      \dfrac{1}{\delta_{DP}|\mathcal{I}|+ |\mathcal{J}|},  & \text{if}\ i \in
      \mathcal{J}.
    \end{cases}
    $$
    }
    
    \State Find $h_{i} \equiv h_{i}(x^k)$, a solution to the
    subproblem \eqref{eq:subprob}
    \State Set $x^{k+1} = x^k + U_ih_{i}$ and
    $v_{(i)}=h_{i}$.
    \EndFor
    
    \State Set $\ell = \ell + c_s$.
    
    \parState{Obtain the set
    $\mathcal{C}(x^{\ell})\subset [m]$, where
    \begin{equation*}
      \mathcal{C}(x^{\ell }) = \left \{ i \ | \ g_j(x^{\ell }) \ge
      \rho_{\alpha}(x^{\ell }), \forall x_j^{\ell} \
      \mbox{s.t.}\  j\  \mbox{belongs to the $i$th-block} \right \},
    \end{equation*}
      \change{with the
    identification function $\rho_{\alpha}(x)$ described in~\eqref{eq:funident}.}
    }
    
    \State Define $\mathcal{J} \subset \mathcal{C}(x^{\ell})$ and
    $\mathcal{I} =[m] \setminus\mathcal{J}$. Set $c_s = \max\{\min\{ \delta_F |\mathcal{I}|,m \}, c_0\}$.
    \Until $\|v\|\le \varepsilon$ or $\ell \, \ge \, \ell_{\max}$
  \end{algorithmic}
\end{algorithm}

In~\cite{LSS19}, the blocks are chosen from a nonuniform probability
distribution. The coordinate blocks are split into two groups: one, denoted by
$\J$, containing a subset of the blocks for which the variables are likely to
activate the bound constraints of the problem, and the other, denoted by $\I$,
with the remaining blocks. The blocks in $\I$ are updated by a probability
distribution that is $\delta_{DP}$ times more likely than the blocks in the
set $\J$, since the objective function value is expected to have a larger
decrease for variables in $\I$ than for those in $\J$. Such a strategy has
proved effective for problems with $\psi(x)=\lambda\|x\|_1$, where $\lambda>0$
is a regularization parameter. Both sets $\I$ and $\J$ are updated along the
method after a pre-established cycle of iterations. The strategy employed to
classify a block of coordinates as active is based upon an identification
function~\cite{FFK98}. Such functions can unveil, in a neighborhood of a
stationary point $x^*$, which constraints are active at $x^*$. 

\change{One example of an identification function that can be used in line~10 
 of Active BCDM, Algorithm~\ref{alg:ActiveBCDM}, was given in
~\cite[Propositions~2 and~3]{LSS19}. To present it, we start by describing the
feasible set $\mathcal{X}$} using the function $g: \mathbb{R}^n \rightarrow
\mathbb{R}^{2n}$, defined by
\[
g_i(x):= \left\{
\begin{array}{ll}
  l_i - x_i, & \text{if} \; 1 \leq i \leq n,\\
  x_i - u_i, & \text{if} \; n+1 \leq i \leq 2n.
\end{array}
\right.
\]
Then, \change{the feasible set $\mathcal{X}$} may be rewritten as
\[
\mathcal{X} = \left\{  x \in \mathbb{R}^n \mid g(x) \leq 0 \right\}.
\]

Given $x \in \mathcal{X}$ and $\beta\in\mathbb{R}_+$, $h^{\beta}(x) \in
\mathbb{R}^n$ is defined as the unique minimizer of
\begin{equation}\label{eq:probsoma}
  \begin{aligned}
    \min_{h \in \R^n} &\quad \sum_{i=1}^m \left( \nabla_i f(x)^{T} h_i + \dfrac{\beta L_i}{2}
    \| h_i\|_{(i)}^2 + \psi_i(x_{(i)}+h_i) \right)\\
    \st        &\quad l_i \le x_{(i)}+h_i \le u_i, \  i \in [m].
  \end{aligned}
\end{equation}
\change{Additionally, let $h(x) \in \mathbb{R}^n$} \change{be defined as $h^\beta(x)$, for $\beta = 1$, that is, the solution of}
\begin{equation}\label{eq:probsomab=1}
  \begin{aligned}
    \min_{h \in \R^n} &\quad \sum_{i=1}^m \left( \nabla_i f(x)^{T} h_i + \dfrac{L_i}{2}
    \| h_i\|_{(i)}^2 + \psi_i(x_{(i)}+h_i) \right)\\
    \st        &\quad l_i \le x_{(i)}+h_i \le u_i, \  i \in [m].
  \end{aligned}
\end{equation}

Notice that problems~\eqref{eq:subprob} and~\eqref{eq:probsomab=1} are closely
related, since the block separability of~\eqref{eq:probsomab=1} implies that the
$i$th block of $h(x)$ is the solution $h_i(x)$ of \eqref{eq:subprob}.

\change{Finally, we draw inspiration from  \change{an idea introduced in~\cite{idenfunc}
 to define a function that identifies the active constraints at the stationary points
 of an optimization problem, namely
\begin{equation}\label{eq:funident}
  \rho_{\alpha}(x) := - \|h(x)\|^{\alpha},
\end{equation}
where $\alpha \in (0, 1)$. Under an error-bound assumption, it was shown
in~\cite[Propositions 2 and 3]{LSS19} that $\rho_\alpha$ is an \change{identification}
function for~\eqref{eq:mainprob} that is highly effective when used in Active
BCDM in the context of $\ell_1$-regularization problems.}}
 \selectlanguage{english}

\section{Parallel Active BCDM}
\label{sec:PCDM}

The parallelization of the Active BCDM Algorithm is not straightforward. An
initial idea would be to distribute the task of the {\tt for} loop to the available
threads, choosing $\tau$ blocks obeying the nonuniform probability
distribution to be updated in parallel. However, \change{to ensure descent of
the objective function, it would be necessary to calculate the Lipschitz
constant for each selected group of blocks and an associated descent
direction. This would greatly affect efficiency.}

To maintain the computational cost of the parallel iterations as close as the
serial \change{version}, both methods \change{use the same} Lipschitz constants
for a fixed block coordinate \change{structure}. Moreover, it is desirable that
both \change{variations should employ} the same nonuniform distribution, which
privileges updating the blocks of inactive variables. To achieve this, we have
\change{adapted the} theoretical framework based upon~\cite{RT2016}. 

We start with a concept presented in~\cite[Section 1.5]{RT2016} which extends
the block separability notion to the smooth part of the objective function of
problem~\eqref{eq:mainprob}. \change{This enables the} characterization of a
wider class of smooth functions $f$ that might benefit from the parallelism.

\begin{definition}\label{def:partsep} 
  The convex and smooth function $f$ is \textbf{partially separable of
  degree~$\omega$} if there exists a finite number of smooth functions~$f_S$
  such that 
  \begin{equation*}
    f(x) = \sum\limits_{S\in\S} f_S(x),
  \end{equation*}
  where $\S$ is a finite collection of nonempty subsets of $[m]$, $f_S$ are
  differentiable convex functions that only depend on blocks $x_{(i)}$ for
  $i\in S$ and 
  $$|S|\le \omega, \ \forall S\in\S.$$
\end{definition}

\change{When} applying block coordinate descent methods to a partially separable
function, \change{we must determine the value of $\omega$ as it influences
the convergence of the methods. Specifically, we aim to find the minimal value
$\omega$ that satisfies Definition~\ref{def:partsep};} however, we may
need to consider larger values for $\omega$ to ensure that the coordinate
blocks of $f_S$, for all $S \in \S$, satisfy a desirable property. 

The partial separability notion is useful \change{because the structure of the
smooth component of relevant problems, like Lasso and $\ell_1$-regularized
logistic regression, allows for a simple} computation of the degree $\omega$;
for further details see~\cite[Section 1.5]{RT2016}.  

 \selectlanguage{english}

\subsection{Descent directions for the method}
\label{sub:PCDMdec}

Throughout this subsection, we assume that the point $x \in \X$, the number
$\tau > 0$ of threads that will work in parallel, together with the disjoint
sets of indices $\I$ and $\J$ that split the groups of coordinate blocks ($\I
\cup \J = [m]$) \change{are all fixed}. 

Let us also recall the concept of a multiset~\cite[Definition
1]{HickmanMultiset}, a generalization of the notion of a set that allows
repetitions of its elements. In the finite case, a multiset can be defined as
a set of tuples, $\{ (x_1, c_1), (x_2, x_2), \ldots (x_k, c_k) \}$, where, for
$i = 1, \ldots, k$, $x_i$ represents the element in the multiset and $c_i$ is
a positive integer representing the cardinality of $x_i$ in the multiset. In
other words, $c_i$ represents the number of times $x_i$ appears repeated in
the multiset. The cardinality of a multiset is the sum of the cardinalities of
its elements. From now on, we will adopt a simplified abuse of notation and
denote a multiset using the notation of a set with the elements $x_i$
appearing $c_i$ times. For example, we will denote the multiset $\{ (1, 2),
(2, 1), (3, 4) \}$ by $\{1, 1, 2, 3, 3, 3, 3 \}$.

Now, let $\B$ be randomly chosen from $[m]$, with cardinality $\tau$, 
allowing possible repetitions, and following the probability distribution of
Algorithm~\ref{alg:ActiveBCDM}. Observe that this definition allows the
algorithm to select a block of coordinates multiple times in a single step.
The objective is to select $\B$ and a constant $\beta \geq 1$, such that
updating $x + \sum_{i\in\B} U_i \hbeta{i}$ produces descent for $F$ in
expectation, where $\hbeta{i}\in\R^{p_i}$ is the solution of
\begin{equation} \label{eq:subprobparal}
  \begin{aligned}
    \min_{h \in \R^{p_i}} &\quad \nabla_i f(x)^{T} h + \dfrac{\beta L_i}{2}
    \| h\|_{(i)}^2 + \psi_i(x_{(i)}+h) - 
    \psi_i(x_{(i)}),\\
    \st        &\quad l \le x+U_ih \le u.
  \end{aligned}
\end{equation}
To reach this goal, a few auxiliary definitions and results are provided.

\begin{definition}[{\bf Intersection keeping all values}]
  Let $I$ and $J$ be a set and a multiset of indexes in $[m]$, respectively.
  The multiset of elements of $J$ that are in $I$ is stated as
  $\sourint{I}{J}$, and the cardinality of such a multiset is denoted by
  $\ourint{I}{J}$.
\end{definition}

For the sake of illustration, let $I=\{1,3,5\}$ and $J=\{1,1,2,3,4\}$. In this
case, $\sourint{I}{J} = \{1,1,3\}$ and $\ourint{I}{J} = 3$, since, considering
the repetitions, the multiset $J$ contains three elements of $I$. 

\begin{proposition}\label{prop:probcond} Let $\B$  be  a multiset randomly chosen
with elements from $[m]$, with cardinality $\tau$, following the probability
distribution of Algorithm~\ref{alg:ActiveBCDM}, and let $S \subset [m]$ be a
nonempty set. For $k\in\mathbb{N}$, with $1\le k \le \tau$, it holds
  \begin{equation}\label{eq:prop-inters1}
    \mathbb{P}(i\in \B \ | \ \ourint{S}{\B} = k) = \begin{cases}
      \dfrac{k\delta_{DP}}{\delta_{DP}|\I \cap S|+|\J \cap S|}, & \text{if} \ i\in \I\cap S;\\
      \\
      \dfrac{k}{\delta_{DP}|\I \cap S|+|\J \cap S|}, & \text{if} \ i\in \J\cap S.\\
    \end{cases}
  \end{equation}
\end{proposition}
\begin{proof}
  First, notice that from the probability distribution of Algorithm~\ref{alg:ActiveBCDM} we have
  \begin{equation}\label{eq:prop-inters2}
    \mathbb{P}(i) = \begin{cases}
      \dfrac{\delta_{DP}}{\delta_{DP}|\mathcal{I}| + |\mathcal{J}|}, &
      \text{if}\ i \in \mathcal{I},\\[12pt]
      \dfrac{1}{\delta_{DP}|\mathcal{I}|+ |\mathcal{J}|},  & \text{if}\ i \in
      \mathcal{J}.
    \end{cases}	 
  \end{equation}
  Additionally, observe that the integer $k\in\mathbb{N}$ such 
  that $1\le k\le \tau$ is constant along the proof. 
  
  Let $i\in \I\cap S$. To reach expression~\eqref{eq:prop-inters1}, we 
  will rest upon the conditional probability formula:
  \begin{equation}\label{eq:prop-inters3}
    \mathbb{P}(i\in \B \ | \ \ourint{S}{\B} = k)=
    \dfrac{\mathbb{P}(i\in \B \ \text{and} \ \ourint{S}{\B} = k)}
    {\mathbb{P}(\ourint{S}{\B} = k)}.
  \end{equation}
  
  For computing the numerator above, let us calculate the following values:
  
  \begin{enumerate}
    \item Probability of the index $i$ to be chosen in a certain position.
    $$p_{\I}\stackrel{\eqref{eq:prop-inters2}}{=}\dfrac{\delta_{DP}}
    {\delta_{DP}|\mathcal{I}| + |\mathcal{J}|}.$$
    
    \item Probability of an element of $S$ to be chosen in a certain position.
    This is obtained by adding up the probabilities of each element of $S$ to
    be in a certain position of $\B$, that is,
    $$p_1=\sum_{i\in \I\cap S}\dfrac{\delta_{DP}}{\delta_{DP}|\mathcal{I}| 
    + |\mathcal{J}|}+\sum_{j\in \J\cap S}\dfrac{1}{\delta_{DP}|\mathcal{I}| 
    + |\mathcal{J}|}=\dfrac{\delta_{DP}|\I\cap S|
    +|\J\cap S|}{\delta_{DP}|\mathcal{I}| + |\mathcal{J}|}.$$
    
    \item Probability of an element that does not belong to $S$ to be chosen
    in a certain position. This is obtained by adding up the probabilities of
    each element that does not belong to $S$ to be in a certain position of
    $\B$, namely,
    \begin{eqnarray*}
      p_2&=&\sum_{i\in \I \setminus S}\dfrac{\delta_{DP}}
      {\delta_{DP}|\mathcal{I}| + |\mathcal{J}|}+\sum_{j\in \J \setminus S}
      \dfrac{1}{\delta_{DP}|\mathcal{I}| + |\mathcal{J}|}\\
      &=&\dfrac{\delta_{DP}(|\I|-|\I\cap S|)+|\J|-|\J\cap S|}
      {\delta_{DP}|\mathcal{I}| + |\mathcal{J}|}.
    \end{eqnarray*}
  \end{enumerate}
  
  Now, to obtain the probability that  $i\in \B$ and  $\ourint{S}{\B} = k$,
  one should organize the previous values. Assume the index~$i$ is chosen in
  the first position, the next $k-1$ positions are elements of~$S$ and the
  last~$\tau - k$ positions are elements of $[m] \setminus S$:
  \begin{equation}\label{eq:prop-inters4}
    p_{\I} \overbrace{p_1 \dots p_1}^{k-1} 
    \underbrace{p_2 \dots p_2}_{\tau-k}.
  \end{equation}
  Notice that \eqref{eq:prop-inters4} amounts to the probability of a specific
  event in which the index $i$ appears in the first position, and the $k$
  first terms are elements of $S$. To account for the remaining events of
  $i\in \B$ and  $\ourint{S}{\B} = k$, we must introduce two additional
  factors. First, $\binom{\tau}{k}$, the number of possibilities of
  distributing the $k$ elements of $S$ in the $\tau$ positions of the random
  multiset $\B$. Second, $\binom{k}{1}$, the number of possibilities of
  setting $i$ among the $k$ elements of $S$ that belong to $\B$. Therefore, 
  \begin{equation}\label{eq:prop-inters5}
    \mathbb{P}(i\in \B \ \text{and} \ \ourint{S}{\B} = k)=
    \binom{k}{1}\binom{\tau}{k}p_{\I} \overbrace{p_1 \dots p_1}^{k-1} 
    \underbrace{p_2 \dots p_2}_{\tau-k}.
  \end{equation}
  
  Now, let us compute the probability of $\ourint{S}{\B} = k$, i.e., the 
  denominator of~\eqref{eq:prop-inters3}. We start by evaluating the 
  probability of the particular event of $\B$ in which the first $k$ indices 
  are elements of $S$, and the last $\tau-k$ elements  do not belong to $S$:
  \begin{equation}\label{eq:prop-inters6}
    \overbrace{p_1 \dots p_1}^{k} \underbrace{p_2 \dots p_2}_{\tau-k}.
  \end{equation}
  
  As in the previous reasoning, the remaining related events must be
  considered. Indeed, the probability~\eqref{eq:prop-inters6} has to be
  multiplied by $\binom{\tau}{k}$, the number of possibilities of distributing
  the $k$ elements of $S$ in the $\tau$ positions of the random multiset $\B$.
  Hence, 
  \begin{equation}\label{eq:prop-inters7}
    \mathbb{P}(\ourint{S}{\B} = k)=\binom{\tau}{k}\overbrace{p_1 \dots p_1}^{k} 
    \underbrace{p_2 \dots p_2}_{\tau-k}.
  \end{equation}
  
  Replacing~\eqref{eq:prop-inters5} and~\eqref{eq:prop-inters7}
  in~\eqref{eq:prop-inters3}, we obtain
  \begin{eqnarray*}
    \mathbb{P}(i\in \B \ | \ \ourint{S}{\B} = k)
    &=&\dfrac{\binom{k}{1}
    \binom{\tau}{k}p_{\I} p_1 ^{k-1} 
    p_2^{\tau-k}}{\binom{\tau}{k}
    p_1^{k} p_2^{\tau-k}}\\[5pt]
    &=&\dfrac{kp_{\I}}{p_1}\\[5pt]
    &=&\dfrac{k\frac{\delta_{DP}}{\delta_{DP}|\mathcal{I}| 
    + |\mathcal{J}|}}{\frac{\delta_{DP}|\I\cap S|
    +|\J\cap S|}{\delta_{DP}|\mathcal{I}|     + |\mathcal{J}|}}\\[5pt]
    &=&\dfrac{k\delta_{DP}}{\delta_{DP}|\I\cap S|+|\J\cap S|}.
  \end{eqnarray*}
  
  Now, assuming that $i\in\J\cap S$, we proceed as before, using the fact that
  the probability of the index $i$ being chosen in a certain position is given
  by
  $$p_{\J}\stackrel{\eqref{eq:prop-inters2}}{=}
  \dfrac{1}{\delta_{DP}|\mathcal{I}| + |\mathcal{J}|}.$$
  
  Arguing as in the previous case, we have 
  \begin{eqnarray*}
    \mathbb{P}(i\in \B \ | \ \ourint{S}{\B} = k)&=&
    \dfrac{\binom{k}{1}\binom{\tau}{k}p_{\J} 
    p_1^{k-1}  p_2^{\tau-k}}
    {\binom{\tau}{k} p_1^{k} 
    p_2 ^{\tau-k}}\\[5pt]
    &=&\dfrac{kp_{\J}}{p_1}\\[5pt]
    &=&\dfrac{k\frac{1}{\delta_{DP}|\mathcal{I}| +
    |\mathcal{J}|}}{\frac{\delta_{DP}|\I\cap S|
    +|\J\cap S|}{\delta_{DP}|\mathcal{I}| + |\mathcal{J}|}}\\[5pt]
    &=&\dfrac{k}{\delta_{DP}|\I\cap S|+|\J\cap S|},
  \end{eqnarray*}
  the proof is complete.
\end{proof}

It is worth mentioning that formulas~\eqref{eq:prop-inters1}
and~\eqref{eq:prop-inters7} are also valid for $k=0$. This case is not
included in Proposition~\ref{prop:probcond} because the theoretical reasoning
to prove~\eqref{eq:prop-inters1} differs from the analysis regarding  $1\le
k\le \tau$. Nevertheless, the aforementioned probability values for $k=0$ are
easily verifiable: $\ourint{S}{\B} = 0$ implies that $S \cap \B = \emptyset$,
and thus, if either $i\in \I\cap S$ or $i\in \J\cap S$ then $i \not \in \B$,
so that $ \mathbb{P}(i\in \B \ | \ \ourint{S}{\B} = 0) =0$. Moreover, it is
not difficult to see that $ \mathbb{P}(\ourint{S}{\B} = 0) =p_2^\tau$.

Before the main result of this section, we will show that,   given $\I$ and $\J$, two fixed and disjoint sets of indices such that $\I \cup \J = [m]$, it is always
 possible to extend the partial separability description of
$f(x) = \sum_{S \in \S} f_S(x)$ in such way that it conforms to
\begin{gather*}
  |\I \cap S_1| = |\I \cap S_2|, \ \forall \ S_1,S_2\in\S, \\
  |\J \cap S_1| = |\J \cap S_2|, \ \forall \ S_1,S_2\in\S,
\end{gather*} 
 the respective degree of partial separability is at most doubled, and it can
be easily computed from the original $\omega$, $|\I|$ and $|\J|$.  

To achieve this, let $\I$ and $\J$ be fixed  
and set $\S^\prime = \emptyset$. Select $S \in \S$. As $|S| \leq \omega$, it
follows that 
\[
| S \cap \I | \leq \min \{ |\I|, \omega \}.
\]
If the equality holds, define $S^\prime = S$. Otherwise, there are $\min
\{|\I|,\omega\} - |S \cap\I|$ block indexes in $\I$ that can be added to $S$,
creating an $S^\prime$ such that $|S^\prime \cap \I| = \min \{ |\I|, \omega \}$.
Similarly, we can define $S^\prime$ from $S$, adding extra elements from $\J$
whenever necessary, such that $|S^\prime \cap \J| = \min \{ |\J|,\omega\}$.
Finally, define $S^\prime = S \cup I \cup J$. As $I \subset \I$ and $J \subset
\J$, these two sets are disjoint, and $S^\prime$ obeys
\[
| S^\prime \cap \I |  = \min \{ |\I|, \omega \}
\quad \text{and} \quad 
|S^\prime \cap \J |  = \min \{ |\J|, \omega \}.
\]
Clearly $S^\prime \supset S$. Now, define $f_{S^\prime} = f_S$. It is true
that the function $f_{S^\prime}$ only depends on variables that are in $S^\prime$, as it
only depends on the variables in $S$. Moreover,
\begin{equation}
| S^\prime | = |S^\prime \cap \I| + |S^\prime \cap \J| = 
\min \{ |\I|, \omega \} + \min \{ |\J|, \omega \} \leq 2\omega.
\end{equation}

Finally, repeat the process to all $S\in\S$ to obtain $\S^\prime =
\{S_1^\prime, S_2^\prime, \ldots, S_{|\S|}^\prime \}$, a partial separability
decomposition of $f$ that has the desired properties. 

The following example illustrates the procedure above. Let $m = 7$, $\I = \{
1, 2\} $ and $\J = \{3,4, 5, 6, 7\}$ and $\S$ be composed of the sets $S_1 =
\{1, 4\}$, $S_2 = \{2\}$, $S_3 = \{ 2, 5, 6 \}$, and $S_4 = \{1 ,3, 7\}$. In
this case, $\omega = 3$, $\min\{|\I|,\omega\}=2$ and $\min\{|\J|,\omega\}=3$. 

For $S_1$, we have $| S_1 \cap \I | = 1 < 2 = \min\{|\I|,\omega\}$, and we can
define $I_1 = \{ 2 \}$. As for $| S_1 \cap \J | = 1 < 3=\min\{|\J|,\omega\}$
and we choose $J_1$ as $\{3,5 \}$. Hence, $S^\prime_1 = S_1 \cup I_1 \cup J_1
= \{1, 2, 3,4, 5 \}$. Similarly, we get $S^\prime_2 = \{1, 2, 3,4,5\}$,
$S^\prime_3 = \{1, 2, 3, 5, 6\}$, $S^\prime_4 = \{ 1, 2, 3,4, 7 \}$, and
$\S^\prime = \{ S^\prime_1, S^\prime_2, S^\prime_3, S^\prime_4 \}$. All
intersections of the elements of $\S^\prime$ with $\I$ and $\J$ have the same cardinality. The partial
separability degree associated with $\S^\prime$ is
$5=\min\{|\I|,\omega\}+\min\{|\J|,\omega\} < 2 \omega$.

We summarize these ideas in the following lemma. Its proof is basically the
discussion that precedes the example.
\begin{lemma} \label{lem:equalsep} 
    Lef $f = \sum_{S \in \S} f_S$ be a partially separable decomposition of $f$
  with degree $\omega$ and  let $\I$ and $\J$ be two fixed and disjoint sets such that $\I \cup \J =[m]$. It is possible to extend the
  sets in $\S$ to form a new decomposition $\S^\prime$ such that, for all $S^\prime \in \S^\prime$,
  \[
    |\I \cap S^\prime|= \min\{|\I|,\omega\} \quad \text{and} \quad
    |\J \cap S^\prime|= \min\{|\J|,\omega\}.
  \]
  Hence, the decomposition $\S^\prime$ has degree $\omega^\prime = \min \{ |\I|, \omega \} + \min \{ |\J|,
  \omega \} \leq 2\omega$.
\end{lemma}

We are now ready to state the main \change{result} of this subsection.
\change{It presents the expected decrease of the objective
function of problem~\eqref{eq:mainprob} whenever blocks of coordinates are
updated following the probability distribution of
Algorithm~\ref{alg:ActiveBCDM}.}

\begin{lemma}\label{lem:espPCDM}
  Consider  $x\in\X$,  $h^{\beta}=[\hbeta{1},\hbeta{2},\dots,\hbeta{m}]^T\in\R^n$, 
  in which $\hbeta{i}$ is the solution of~\eqref{eq:subprobparal}, 
  and a random multiset $\B$ as in Proposition~\ref{prop:probcond}, with $|\B|=\tau$. 
  Additionally, suppose that  $f$ is a partially separable function of degree $\omega$, and let $\I$ and $\J$ be two fixed and disjoint sets of indices such that $\I \cup \J=[m]$. Then,
  \begin{small}
    \begin{eqnarray}\label{eq:espdecr-smo}
      \mathbb{E}\left[f\left(x+
      \sum\limits_{i\in \B}U_{i}\hbeta{i}\right)  \right] 
      & \le & f(x) + \sum\limits_{i\in\I} 
      \dfrac{\tau\delta_{DP}}{q}\left(\nabla_i f(x)^T\hbeta{i} + 
      \dfrac{\beta L_i}{2} \|\hbeta{i}\|^2_{(i)} \right)\nonumber\\
      &  & + \sum\limits_{i\in\J} 
      \dfrac{\tau}{q}\left(\nabla_i f(x)^T\hbeta{i} + 
      \dfrac{\beta L_i}{2} \|\hbeta{i}\|^2_{(i)} \right),
    \end{eqnarray}
  \end{small}
  with
  $\beta=\dfrac{(\tau-1)(\delta_{DP}\min\{|\I|,\omega\}+\min\{|\J|,\omega\})}{q}+1$,
  and $q=\delta_{DP}|\mathcal{I}| + |\mathcal{J}|$.
\end{lemma}
\begin{proof} 
Without loss of generality, we start by applying  Lemma~\ref{lem:equalsep} to build a new set $\S^\prime$ such that $f(x)= \sum_{S^\prime \in \S^\prime} f_{S^\prime}(x)$, which satisfies
\begin{equation}\label{eq:eqsets}
\begin{cases}
|\I \cap S^\prime| = \min\{|\I|,\omega\}, \ \forall \ S^\prime\in\S^\prime, \\
  |\J \cap S^\prime| = \min\{|\J|,\omega\}, \ \forall \ S^\prime\in\S^\prime   , 
\end{cases}  
\end{equation}     
with partial separability degree $\omega^\prime = \min \{ |\I|, \omega \} + \min \{ |\J|,
  \omega \}$.

Now, to simplify the notation, let
$H:=\sum\limits_{i\in \B}U_{i}\hbeta{i}$,
  and define the functions
$$\phi(h):=f(x+h)-f(x)-\nabla f(x)^Th,$$
$$\phi_{S^\prime}(h):=f_{S^\prime}(x+h)-f_{S^\prime}(x)-\nabla f_{S^\prime}(x)^Th,
  \forall \ S^\prime\in\S^\prime,$$
  in which $f_{S^\prime}, \ S^\prime\in\S^\prime$, are those functions that come from
  the partial separability of $f$ (cf. Definition~\ref{def:partsep}).
  Since
  \begin{eqnarray*}
    \mathbb{E}\left[\phi(H) \right]  & = & 
    \mathbb{E}\left[f(x+H)-f(x)-\nabla f(x)^TH \right]\\
    & = & \mathbb{E}\left[f(x+H) \right]-f(x)
    -\sum\limits_{i=1}^m\nabla_i f(x)^T\hbeta{i}\mathbb{P}(i\in\B)\\
    & = & \mathbb{E}\left[f(x+H)  \right] - f(x) 
    -\sum\limits_{i\in\I} \dfrac{\tau\delta_{DP}}{q}\nabla_i f(x)^T\hbeta{i} +
    \nonumber\\
    &   & - \sum\limits_{i\in\J} \dfrac{\tau}{q}\nabla_i f(x)^T\hbeta{i},
  \end{eqnarray*}
  to obtain the desired result, it is enough to ensure that 
  \begin{equation}\label{eq:2.2.1}\mathbb{E}\left[\phi(H)  \right] \le 
    \sum\limits_{i\in\I} \dfrac{\tau\delta_{DP}}{q}\left(
    \dfrac{\beta L_i}{2}\|\hbeta{i}\|_{(i)}^2\right)
    + \sum\limits_{i\in\J} \dfrac{\tau}{q}
    \left(\dfrac{\beta L_i}{2}\|\hbeta{i}\|_{(i)}^2\right),
  \end{equation}
  for some $\beta\in\mathbb{R}$.
  
  Resting upon the sets  $S^\prime\in\S^\prime$ from Lemma~\ref{lem:equalsep},
  the expected value of function  $\phi$ may be expressed as the 
  following sum of conditional expectations:
  \begin{eqnarray}\label{eq:2.2.2}
    \mathbb{E}\left[ \phi\left(H\right) 
    \right] & = & \sum_{k=0}^{\tau}\sum_{S^\prime\in \S^\prime} \mathbb{P}(\ourint{S^\prime}{\B}=k) 
    \mathbb{E}\left[ \phi_{S^\prime}\left(H\right) \ | \ \ourint{S^\prime}{\B}=k \right] .
  \end{eqnarray}
  
  The relations in~\eqref{eq:eqsets} guarantee that the cardinalities $|\I\cap S^\prime|$ and $|\J\cap S^\prime|$ do not depend on the sets $S^\prime\in \S^\prime$. 
  From~\eqref{eq:prop-inters7}, the probability $\mathbb{P}(\ourint{S^\prime}{\B}=k)$
  is also independent on $S^\prime\in \S^\prime$ for all fixed $k$, $0\le k\le \tau$. Therefore,
  rewriting~\eqref{eq:2.2.2} we obtain
  \begin{eqnarray}\label{eq:2.2.3}
    \mathbb{E}\left[ \phi\left(H\right) 
    \right] & \le & \sum_{k=0}^{\tau}\mathbb{P}(\ourint{S^\prime}{\B}=k)
    \sum_{S^\prime\in \S^\prime} 
    \mathbb{E}\left[  
    \phi_{S^\prime}(H) \ \bigg| \  \ourint{S^\prime}{\B}=k \right].
  \end{eqnarray}
  
  For $k=0$, and for all $S^\prime\in\S^\prime$, the expected value of $\phi_{S^\prime}$ in~\eqref{eq:2.2.3} may be expressed as
  \begin{eqnarray}\label{eq:2.2.4}
    \mathbb{E}\left[ \phi_{S^\prime}\left(H\right) \ | \   \ourint{S^\prime}{\B}=0 \right] 
    & = & 
    \mathbb{E}\left[ \phi_{S^\prime}\left(\sum\limits_{i\in\sourint{S^\prime}{\B}}
    U_{i}\hbeta{i}\right) \ \bigg| \   \ourint{S^\prime}{\B}=0 \right]  
    \nonumber \\
    & = & 
    \mathbb{E}\left[ \phi_{S^\prime}\left(\sum\limits_{i\in \emptyset} 
    U_{i}\hbeta{i}\right) \right]
    \nonumber \\
    & = & 
    \phi_{S^\prime}( 0)
    \nonumber \\
    & = & f_{S^\prime}(x+0)-f_{S^\prime}(x)-\nabla f_{S^\prime}(x)^T{ 0}
    \nonumber \\
    & = & 0.
  \end{eqnarray}
  
  For each fixed $k$, $1\le k \le \tau$, from the convexity of the 
  function $\phi_{S^\prime}$ we have
  \begin{align}\label{eq:2.2.5}
    \mathbb{E}\left[ \phi_{S^\prime}\left(H\right) \ | \   \ourint{S^\prime}{\B}=k \right] & =   
    \mathbb{E}\left[ \phi_{S^\prime}\left(\dfrac{1}{k}\sum\limits_{i\in\sourint{S^\prime}{\B}}kU_{i}\hbeta{i}\right) \ \bigg| \  \ \ourint{S^\prime}{\B}=k \right]
    \nonumber \\
    & \le  \mathbb{E}\left[\dfrac{1}{k}\sum\limits_{i\in\sourint{S^\prime}{\B}}  
    \phi_{S^\prime}\left(kU_{i}\hbeta{i}\right) \ \bigg| \  \ourint{S^\prime}{\B}=k \right] \nonumber \\
    & = \dfrac{1}{k}\mathbb{E}\left[\sum\limits_{i\in\sourint{S^\prime}{\B}}  
    \phi_{S^\prime}\left(kU_{i}\hbeta{i}\right) \ \bigg| \  \ourint{S^\prime}{\B}=k \right] \nonumber \\
    &=  \dfrac{1}{k}\sum\limits_{i\in S^\prime}  
    \phi_{S^\prime}\left(kU_{i}\hbeta{i}\right)\mathbb{P}(i\in \B \ | \ \ourint{S^\prime}{\B}=k) \nonumber \\
    & \stackrel{\eqref{eq:prop-inters1}}{= }  \dfrac{1}{k}\left[\sum\limits_{i\in S^\prime\cap\I}  
    \phi_{S^\prime}\left(kU_{i}\hbeta{i}\right)\dfrac{k\delta_{DP}}{\delta_{DP}|\I\cap S^\prime|+|\J\cap S^\prime|}+\right. \nonumber \\
    &  \qquad + \left.\sum\limits_{i\in S^\prime\cap\J}  
    \phi_{S^\prime}\left(kU_{i}\hbeta{i}\right)\dfrac{k}{\delta_{DP}|\I\cap S^\prime|+|\J\cap S^\prime|}\right] \nonumber \\
    & =   z\left[\sum\limits_{i\in S^\prime\cap\I}  
    \delta_{DP}\phi_{S^\prime}\left(kU_{i}\hbeta{i}\right)+\sum\limits_{i\in S^\prime\cap\J}  
    \phi_{S^\prime}\left(kU_{i}\hbeta{i}\right)\right],
  \end{align}
  for all $S^\prime\in\S^\prime$ and $z:=\dfrac{1}{\delta_{DP}|\I\cap S^\prime|+|\J\cap S^\prime|}$.
  
  Applying expressions~\eqref{eq:2.2.4} and~\eqref{eq:2.2.5} 
  to \eqref{eq:2.2.3} yields
  \begin{align}\label{eq:2.2.6}
    \mathbb{E}\left[ \phi\left(H\right) 
    \right] & \le  \sum_{k=1}^{\tau}\mathbb{P}(\ourint{S^\prime}{\B}=k)
    \sum_{S^\prime\in \S^\prime} 
    z\left[\sum\limits_{i\in S^\prime\cap\I}  
    \delta_{DP}\phi_{S^\prime}\left(kU_{i}\hbeta{i}\right)+ \right. \nonumber\\
    &\qquad \qquad \qquad \qquad \qquad \qquad + \left. \sum\limits_{i\in S^\prime\cap\J}  
    \phi_{S^\prime}\left(kU_{i}\hbeta{i}\right)\right]\nonumber\\
    & = z\sum_{k=1}^{\tau}\mathbb{P}(\ourint{S^\prime}{\B}=k)
    \left[\sum\limits_{i\in\I}  
    \delta_{DP}\phi\left(kU_{i}\hbeta{i}\right)+\sum\limits_{i\in\J}  
    \phi\left(kU_{i}\hbeta{i}\right)\right]\nonumber\\
    &\le z\sum_{k=1}^{\tau}\mathbb{P}(\ourint{S^\prime}{\B}=k)
    \left[\sum\limits_{i\in\I}  
    \delta_{DP}\dfrac{L_i}{2}\|k\hbeta{i}\|_{(i)}^2+\sum\limits_{i\in\J}  
    \dfrac{L_i}{2}\|k\hbeta{i}\|_{(i)}^2\right]\nonumber\\
    &= z\sum_{k=1}^{\tau}\mathbb{P}(\ourint{S^\prime}{\B}=k) k^2
    \left[\sum\limits_{i\in\I}  
    \delta_{DP}\dfrac{L_i}{2}\|\hbeta{i}\|_{(i)}^2+\sum\limits_{i\in\J}  
    \dfrac{L_i}{2}\|\hbeta{i}\|_{(i)}^2\right].
  \end{align}
  
  Now, from ~\cite[(24)]{RT2016}, it holds
  \begin{eqnarray}\label{eq:2.2.7}
    \mathbb{E}\left[ \left(\ourint{S^\prime}{\B} \right)^2  \right] &=&
    \sum_{k=1}^{\tau}\mathbb{P}(\ourint{S^\prime}{\B}=k) k^2 \nonumber \\
    &=& \tau p_1(\tau p_1+p_2)=\tau p_1((\tau-1) p_1+1),
  \end{eqnarray}
  since $p_2=1-p_1$,  with $p_1$ and $p_2$ as established in Proposition~\ref{prop:probcond}.

  Through the definition of $p_1$ and~\eqref{eq:eqsets}, we obtain
  \begin{equation}\label{eq:2.p1}
  p_1=\dfrac{\delta_{DP}|\I \cap S^\prime| + |\J \cap S^\prime|}{\delta_{DP}|\I |+|\J |}= \dfrac{\delta_{DP}\min\{|\I|,\omega\}+\min\{|\J|,\omega\}}{q}. 
  \end{equation}
  
  Using~\eqref{eq:2.2.7},~\eqref{eq:2.p1}  and $q = \delta_{DP} |\I | + |\J|$, inequality~\eqref{eq:2.2.6} becomes
  \begin{eqnarray}\label{eq:2.2.8}
    \mathbb{E}\left[ \phi\left(H\right)  \right] 
    & \le & z\tau p_1((\tau-1) p_1+1)
    \left[\sum\limits_{i\in\I}  
    \delta_{DP}\dfrac{L_i}{2}\|\hbeta{i}\|_{(i)}^2+\sum\limits_{i\in\J}  
    \dfrac{L_i}{2}\|\hbeta{i}\|_{(i)}^2\right]\nonumber \\
    & = & \dfrac{\tau}{q} ((\tau-1) p_1+1)
    \left[\sum\limits_{i\in\I}  
    \delta_{DP}\dfrac{L_i}{2}\|\hbeta{i}\|_{(i)}^2+\sum\limits_{i\in\J}  
    \dfrac{L_i}{2}\|\hbeta{i}\|_{(i)}^2\right]\nonumber \\
    & \stackrel{\eqref{eq:2.p1}}{=} & \dfrac{\tau \delta_{DP}}{q} \sum\limits_{i\in\I} 
    \dfrac{\beta L_i}{2}\|\hbeta{i}\|^2_{(i)} + \dfrac{\tau}{q} 
    \sum\limits_{i\in\J} \dfrac{\beta L_i}{2}\|\hbeta{i}\|^2_{(i)},
  \end{eqnarray}
  with $\beta=\dfrac{(\tau-1)(\delta_{DP}\min\{|\I|,\omega\}+\min\{|\J|,\omega\} )}
  {q}+1$.
  
  As~\eqref{eq:2.2.8} is exactly inequality~\eqref{eq:2.2.1}, the 
  proof is complete.
\end{proof}

Owing to Lemma~\ref{lem:espPCDM}, an expression for the 
expected decrease of function  $F=f+\psi$ may be attained by 
updating simultaneously $\tau$ blocks of coordinates, with each 
of these blocks obeying the probability distribution described 
in~Algorithm~\ref{alg:ActiveBCDM}.

Let $\B$ be a multiset of independent and identically distributed (i.i.d.) 
random variables, as previously defined,
and let $\hbeta{i}\in\R^{p_i}, i \in [m]$ be blocks of solution vectors
of the subproblems~\eqref{eq:subprobparal}.  The expected value of $F$
may be expressed as
\begin{equation}\label{eq:espvalue1}
  \mathbb{E}\left[F\left(x+
  \sum\limits_{i\in \B}U_{i}\hbeta{i}\right)\right] = 
  \mathbb{E}\left[f\left(x+
  \sum\limits_{i\in \B}U_{i}\hbeta{i}\right)\right] +
  \mathbb{E}\left[\psi\left(x+
  \sum\limits_{i\in \B}U_{i}\hbeta{i}\right)\right].
\end{equation} 

Due to the separable structure of $\psi(x)$, we have
\begin{eqnarray}\label{eq:espvalue2}
  \mathbb{E}\left[\psi\left(x+
  \sum\limits_{i\in \B}U_{i}\hbeta{i}\right)\right] 
  & = & \mathbb{E}\left[\sum\limits_{i\in\B}
  \psi_i(x_{(i)}+\hbeta{i})+\sum\limits_{i\notin\B}
  \psi_i(x_{(i)})\right]\nonumber\\ 
  & = & \mathbb{E}\left[\sum\limits_{i\in\B}
  (\psi_i(x_{(i)}+\hbeta{i})-\psi_i(x_{(i)}))+
  \sum\limits_{i=1}^m \psi_i(x_{(i)})\right]\nonumber\\
  & = & \sum\limits_{i\in\I}
  \dfrac{\tau\delta_{DP}}{q}\psi_i(x_{(i)}+\hbeta{i})+
  \left(1-\dfrac{\tau\delta_{DP}}{q}\right)\psi_i(x_{(i)})+\nonumber\\
  &   & +
  \sum\limits_{i\in\J}\dfrac{\tau}{q}\psi_i(x_{(i)}+\hbeta{i})
  +\left(1-\dfrac{\tau}{q}\right)\psi_i(x_{(i)}).
\end{eqnarray}

Using~\eqref{eq:espdecr-smo} and~\eqref{eq:espvalue2} 
in equality~\eqref{eq:espvalue1}, we obtain
\begin{eqnarray}\label{eq:espvalue3}
  \mathbb{E}\left[F\left(x+
  \sum\limits_{i\in \B}U_{i}\hbeta{i}\right)\right] 
  & \le & F(x) +  \sum\limits_{i\in\I} \dfrac{\tau\delta_{DP}}{q}
  \left( \nabla_{i} f(x)^T\hbeta{i} + 
  \dfrac{\beta L_i}{2}\|\hbeta{i}\|^2_{(i)} + \right.\nonumber\\
  & & + \left. \psi_i(x_{(i)}+\hbeta{i}) - \psi_i(x_{(i)}) 
  \right) + \sum\limits_{i\in\J} \dfrac{\tau}{q}
  \left( \nabla_{i} f(x)^T\hbeta{i} + \right. \nonumber\\
  & & + \left. \dfrac{\beta L_i}{2}\|\hbeta{i}\|^2_{(i)}
  + \psi_i(x_{(i)}+\hbeta{i}) - \psi_i(x_{(i)}) \right).
\end{eqnarray}
This ensures that solving $\tau$ subproblems~\eqref{eq:subprobparal} from a
fixed vector $x\in\X$, as previously described, the direction $\sum_{i\in
\B}U_{i}\hbeta{i}(x)$ is expected to induce a decrease in $F$. Thus, we have
shown to be theoretically sound to base an algorithm on this type of updating
to compute descent directions for problem~\eqref{eq:mainprob}. 

\change{This leads directly to} Algorithm~\ref{alg:ActivePCDM} (Parallel Active
BCDM). \change{It is} a parallel version of Algorithm~\ref{alg:ActiveBCDM},
supported by the theoretical development of this subsection. Note that the
algorithm does not depend on the specific sets in the partial separability
decomposition of the objective $\S$. It only depends on the partial
separability degree of $f$ and the sizes of the sets $\I$ and $\J$. This
explains the formula for $\beta$ appearing in the method.

\change{Algorithm~\ref{alg:ActivePCDM} also uses other
parameters as inputs. The parameter $\tau$ is the number of threads available
for the method; $\delta_{F}$ controls the number of iterations the algorithm
performs before reevaluating the identification function. Furthermore,
$\delta_{DP}$ determines how many times the probability of selecting inactive
blocks should be higher than that of active blocks.}

\change{In the outer loop (starting at line 7 of the algorithm), it
selects $\tau$ coordinate blocks according to the probability distribution
$\mathbb{P}(i)$ and update their coordinates in parallel using our proposed
descent direction. At the end of each cycle of $\gamma$ parallel iterations,
the sets of active and inactive coordinate blocks are updated using the
identification function (as shown in lines 15 and 16 of the algorithm).}

\begin{algorithm}[!ht]
  \caption{Parallel BCDM with Identification of Active Variables (Parallel Active BCDM)}
  \label{alg:ActivePCDM}
  
  \begin{algorithmic}[1]
    
    \State Choose an initial point $l\le x^0 \le u$, \change{a block separation of $\psi$ function described as $x = \sum_{i=1}^{m}U_ix_{(i)}$,}  an initial cycle size
    $c_0$, the parameters $\ell_{\max} \in \mathbb{N}$,
    $\varepsilon\in\R_{+}$, $\omega$ the partial separable degree of $f$, the number of threads $\tau$ and two natural numbers
    $\delta_F, \delta_{DP}$. Initialize a vector $v\in\mathbb {R}^n$ with
    $2\varepsilon$ in all positions, the sets $\mathcal{I}=[m]$,
    $\mathcal{J}=\emptyset$, the initial cycle size $c_s=c_0$, and the
    counters $\ell, j = 0$.
    \State \change{Calculate the Lipschitz constants of the gradient by blocks satisfying~\eqref{eq:lipcondblk}.}
    \Repeat
    \State $q=\delta_{DP}|\I|+|\J|;$
    \State $\gamma \longleftarrow \lfloor\frac{c_s}{\tau}\rfloor;$
    \State $\beta=\dfrac{(\tau-1)(\delta_{DP}\min\{|\I|,\omega\}+\min\{|\J|,\omega\})}{q}+1$,
  with $q=\delta_{DP}|\mathcal{I}| + |\mathcal{J}|$; 
    \For {$k = (j - 1) \gamma +1, \ldots, j \gamma$}
    
    \parState{Choose a multiset $\B^k\subset [m]$ 
    of $\tau$ blocks of coordinates, where each element of 
    $\B^k$ satisfies the probability distribution
    $$
    \mathbb{P}(i) = \begin{cases}
      \dfrac{\delta_{DP}}{\delta_{DP}|\mathcal{I}| + |\mathcal{J}|}, &
      \text{if}\ i \in \mathcal{I},\\[12pt]
      \dfrac{1}{\delta_{DP}|\mathcal{I}|+ |\mathcal{J}|},  & \text{if}\ i \in
      \mathcal{J}.
    \end{cases}
    $$
    }
    \ParFor {$i \in \B^k$} 
    \State Find $\hbeta{i} \equiv \hbeta{i}(x^k)$, a solution to the
    subproblem \eqref{eq:subprobparal}
    \State Set $x^{k+1} = x^k + U_i\hbeta{i}$ and
    $v_{(i)}=\hbeta{i}$.
    \EndParFor
    \EndFor
    
    \State Set $j = j + 1$   and $\ell = \ell + \gamma\tau$.
    
    \parState{Obtain the set
    $\mathcal{C}(x^{\ell})\subset [m]$, where
    \begin{equation*}
      \mathcal{C}(x^{\ell }) = \left \{ i \ | \ g_j(x^{\ell }) \ge
      \rho_{\alpha}(x^{\ell }), \forall x_j^{\ell} \
      \mbox{s.t.}\  j\  \mbox{belongs to the $i$th-block} \right \},
    \end{equation*}
      \change{with the
    identification function $\rho_{\alpha}(x)$ described in~\eqref{eq:funident}.}
    }
    
    \State Define $\mathcal{J} \subset \mathcal{C}(x^{\ell})$ and
    $\mathcal{I} = [m] \setminus \mathcal{J}$. Set  $c_s = \max\{\min\{ \delta_F |\mathcal{I}|,m \}, c_0\}$.
    \Until $\|v\|\le \varepsilon$ or $\ell \, \ge \, \ell_{\max}$
  \end{algorithmic}
\end{algorithm}
 
\selectlanguage{english}

\subsection{\change{Convergence results}}
\label{sub:convergence}

\change{Using the notation above, the global
convergence and the complexity analysis} of the
Algorithm~\ref{alg:ActivePCDM} 
\change{are obtained as follows.}

\change{\begin{definition}\label{def:func-G}
    Let $Q(h,x,\gamma):\R^n\times \R^n \times \R_+ \rightarrow \R$ and $G(h,x,\gamma):\R^n\times \R^n \times \R_+ \rightarrow \R$ be the functions
\[ Q(h,x,\gamma) :=  f(x) + \nabla f(x)^Th + \dfrac{\gamma}{2}
    \|h\|^2_2 + \psi(x+h)\]
    and
\[ G(h,x,\gamma) :=  \nabla f(x)^Th + \dfrac{\gamma}{2}
    \|h\|^2_2 + \psi(x+h) - \psi(h). \]
    Additionally, we define $G_i(h_i,x,\gamma):\R^{p_i}\times \R^{n} \times \R_+ \rightarrow \R$ as
\[ G_i(h_i,x,\gamma) :=  \nabla_i f(x)^Th_i + \dfrac{\gamma}{2}
    \|h_i\|^2_2 + \psi_i(x_{(i)}+h_i) -\psi_i(x_{(i)}) \]
    and, for $i \in [m]$.
\end{definition}}

\change{Notice \change{that with this definition}
\begin{equation}\label{eq:sivector}
        \begin{aligned}
            h_i^{\beta}(x)=\argmin_{h_i \in \R^{p_i}} &\quad G_i(h_i,x,\beta) \\
            \st        &\quad l_i \le x_{(i)}+h_i \le u_i
        \end{aligned}    
\end{equation}
and
\begin{equation}\label{eq:sdx}
    \begin{aligned}
        h^{\beta}(x)=\argmin_{h \in \R^{n}} &\quad G(h,x,\beta).\\
        \st        &\quad l \le x+h \le u
    \end{aligned}    
\end{equation}}

\begin{definition}\label{def:statonarity}
    \change{$\X^*$ is the set of minimizers of problem~\eqref{eq:mainprob}}.
\end{definition}

\begin{assumption}\label{ass:lim-func}
    The set $\ X^*$ is not empty, and the minimum value of the problem is
$F^*=\min_{x\in\X} F(x)$.
\end{assumption}

The error bound condition stated next comes from~\cite[EB condition
(47)]{Tseng10}. This condition can be verified in some situations. \change{For} example,
when $f$ is strongly convex and has Lipschitz continuous gradient. Another
is whenever $f$ is a quadratic function (even nonconvex) and $\psi$ is a
polyhedral function, see~\cite{Tseng10,TY08} for more examples.

\begin{assumption}\label{ass:error-bound} 
    
    Let $\varrho \ge F^*$, suppose there exist $\epsilon_1,\epsilon_2>0$ such
    that
    \change{\[ 
        dist(x,\X^*) \le \epsilon_2 \| h (x) \|, \ \text{whenever} \ 
        F(x)\le\varrho, \ \| h(x)\|\le \epsilon_1,
    \]
    where $dist(x,\X^*)=\min_{y\in \X^*} \|x-y\|$ and $h(x)$ is the solution of~\eqref{eq:probsomab=1}}.
\end{assumption} 

\begin{assumption}\label{ass:lips-contin}
    The gradient of $f$ is globally Lipschitz continuous with respect 
    to the Euclidean norm, i.e., there exists $L_f>0$ such that
\[ \|\nabla f(y) - \nabla f(x)\| \le L_f\|y-x\|, 
    \ \forall x, y \in \R^n.  \]
\end{assumption}

A classical consequence of the last assumption is described next.

\begin{corollary} \label{cor:quad-sup-bound} 
    \cite[Lemma 4.1.12]{Dennis1996} 
    Under Assumption~\ref{ass:lips-contin},
    \[ 
        |f(y) - f(x) - \nabla f(x)^T(y-x)| \le \dfrac{L_f}{2}
        \|y-x\|^2. 
    \]
\end{corollary}

\begin{corollary}\label{cor:stationarity}
    A point $x^*\in\X^*$ if, and only if, \change{$h_i(x^*) = 0$} for all $i \in [m]$,
    with \change{$h(x)$ the solution of~\eqref{eq:probsomab=1}}.
\end{corollary}
\begin{proof}
    Follows as an immediate consequence of~\cite[Lemma 2]{LSS19}.
\end{proof}

We now present some auxiliary results that will be useful \change{below.}

\begin{lemma}\label{lem:pre-converg1}
    Let $\{x^k\}$ be the sequence generated by Algorithm~\ref{alg:ActivePCDM}
    (Parallel Active BCDM) and $\B^k \subset [m]$ be a multiset of randomly
    generated i.i.d. indices that controls the choice of $\tau$ blocks of
    coordinates at the $k$-th iteration of the method, where each element of
    the multiset $\B^k$ \change{is} chosen by the probability distribution of
    Algorithm~\ref{alg:ActivePCDM}. Then,
    \change{\[ 
        F(x^k) - \mathbb{E}\left[F(x^{k+1})| \B^k \right] 
        \ge \dfrac{1}{2m\delta_{DP}}
        \|h(x^k)\|^2,  
    \] }
    where 
    $\mathbb{E}\left[F(x^{k+1})| \B^k\right]   =  
    \mathbb{E}\left[F\left(x^k+\sum\limits_{i\in\B^k} 
    U_i\hbetak{i}(x^k) \right) \right]$, 
with $\hbetak{i}(x^k)$ solution of~\eqref{eq:subprobparal}, $i\in [m]$,
    and \change{$h(x)$ the solution of~\eqref{eq:probsomab=1}}.
\end{lemma}
\begin{proof}
    Due to the dynamics of Algorithm~\ref{alg:ActivePCDM} and
    Equation~\eqref{eq:espvalue3}, we have
    \change{\begin{eqnarray}\label{eq:lem1-eq1}
        \mathbb{E}\left[F(x^{k+1})| \B^k\right]  
        & \le & F(x^k) +  \sum\limits_{i\in\I} \dfrac{\tau\delta_{DP}}{q}
        G_i(\hbetak{i}(x^k),x^k,\beta_k) + \nonumber\\
        &    & + \sum\limits_{i\in\J} 
        \dfrac{\tau}{q} G_i(\hbetak{i}(x^k),x^k,\beta_k)\nonumber\\
        & \le & F(x^k) +  \sum\limits_{i=1}^m \dfrac{\tau\delta_{DP}}{q}
        G_i(\hbetak{i}(x^k),x^k,\beta_k).
    \end{eqnarray}}
    
    First, we use the fact that $|\I|+|\J|=m$, despite of the change of the
    sets $\I$ and $\J$  along the iterations. Second, since
    \change{$G_i(\hbetak{i}(x^k),x^k,\beta_k) \le G_i(h_i^{L_i\beta_k}(x^k),x^k,L_i\beta_k)\le
    0$, where $\hbetak{i}$ and $h_i^{L_i\beta_k}(x^k)$ 
    are the minimizers of the
    subproblem~\eqref{eq:sivector}, for all $i \in [m]$, each one with your specific $\beta$. Putting all this together, we get
    \begin{eqnarray}\label{eq:lem1-eq1.2}
        &&\dfrac{\tau\delta_{DP}}{q} \ge \dfrac{1}{m\delta_{DP}} \nonumber\\
        &\Leftrightarrow& \dfrac{\tau\delta_{DP}}{q}G_i(\hbetak{i}(x^k),x^k,\beta_k) \le
        \dfrac{1}{m\delta_{DP}}G_i(h_i^{L_i\beta_k}(x^k),x^k,L_i\beta_k)
        \nonumber\\
        &\Leftrightarrow& \dfrac{\tau\delta_{DP}}{q}G_i(\hbetak{i}(x^k),x^k,\beta_k) \le
        \dfrac{1}{m\delta_{DP}}G_i(h_i^{\overline{\beta}}(x^k),x^k,
        \overline{\beta}), 
    \end{eqnarray}}
    with 
    \begin{equation}\label{eq:betabarra}
        \overline{\beta}:=\beta_{\max}\lambda_{\max}L_{\max},
    \end{equation}
    where $\beta_{\max}=2(\tau-1)\omega(\delta_{DP}+1)/m$, $\lambda_{\max}$ is the largest eigenvalue of the matrix $B\in\R^{n\times n}$, and $L_{\max}=\max_{1\le i \le m}\{ L_i\}$. 
    
    Through the expression~\eqref{eq:lem1-eq1.2}, the inequality~\eqref{eq:lem1-eq1} admits the 
    following upper bound
   \change{ \begin{align}\label{eq:lem1-eq2}
        \mathbb{E}\left[F(x^{k+1})| \B^k\right] 
        & \stackrel{\eqref{eq:lem1-eq1.2}}{\le}  F(x^k) +  \dfrac{1}{m\delta_{DP}}\sum\limits_{i=1}^m G_i(h_i^{\overline{\beta}}(x^k),x^k,\overline{\beta}) \nonumber \\
        & =  F(x^k) +  \dfrac{1}{m\delta_{DP}}G(h^{\overline{\beta}}(x^k),x^k,\overline{\beta})\nonumber\\
        & =  \left(1-\dfrac{1}{m\delta_{DP}}\right)F(x^k) +  \dfrac{1}{m\delta_{DP}}\left( Q(h^{\overline{\beta}}(x^k),x^k,\overline{\beta})\right).
    \end{align}}

    From the relationship~\eqref{eq:lem1-eq2} and because the function
    $Q(\cdot,x^k,\overline{\beta})$ is strongly convex with respect to the
    norm $\|\cdot\|_2$, it holds that
       \change{\begin{eqnarray}\label{eq:lem1-eq3}
        F(x^k) - \mathbb{E}\left[F(x^{k+1})| \B^k\right] 
        &\ge&  \dfrac{1}{m\delta_{DP}} F(x^k) - 
        \dfrac{1}{m\delta_{DP}}Q(h^{\overline{\beta}}(x^k),x^k,
        \overline{\beta})\nonumber\\
        &  =  &  \dfrac{1}{m\delta_{DP}} 
        (Q(0,x^k,\overline{\beta}) - Q(h^{\overline{\beta}}(x^k),x^k,
        \overline{\beta})\nonumber\\
        & \ge  &   \dfrac{\overline{\beta}}{2m\delta_{DP}} 
        \|h^{\overline{\beta}}(x^k)\|^2.
    \end{eqnarray}}
    
    Assuming $\overline{\beta}\ge 1$, which may be achieved without loss of
    generality by increasing the Lipschitz constants $L_i$ or choosing
    conveniently the matrix $B$, and using~\cite[Lemma 4]{PN15}, we see that
       \change{\begin{equation}\label{eq:lem1-eq4}
        \overline{\beta}\|h^{\overline{\beta}}(x^k)\|^2\ge \|h(x^k)\|^2.
    \end{equation}}
    
    Putting the expression~\eqref{eq:lem1-eq3} and~\eqref{eq:lem1-eq4}
    together concludes the proof.
\end{proof}

\begin{lemma}\label{lem:pre-converg2} 
    Let $x\notin \X^*$. Then, there exist $\gamma,\delta>0$ such that, for all
    feasible $y \in \mathbb{B}(x,\delta)$, we get  \change{$\|h(y)\| \ge \gamma
    \|h(x)\|$, with $h(\cdot)$ the solution of~\eqref{eq:probsomab=1}}.
\end{lemma}
\begin{proof}
    Just apply the same arguments of~\cite[Lemma 4]{LSS19}. 
\end{proof}

\change{The following result establishes a global convergence property for Algorithm~\ref{alg:ActivePCDM}. We show that every sequence generated by the algorithm converges to a point $x^* \in \X$ by applying Corollary~\ref{cor:stationarity}, thus guaranteeing that $h_i(x^*) = 0$ for all $i \in [m]$, which is equivalent to $h(x^*)=\sum_{i=1}^m U_ih_i(x^*)=0$.}

\begin{theorem}\label{lem:pre-converg3}
    Let $\{x^k\}$ be an infinite sequence generated by
    Algorithm~\ref{alg:ActivePCDM} (Parallel Active BCDM) under the hypotheses of Lemma~\ref{lem:pre-converg1}. Then, \change{$\|h(x^k)\|\rightarrow
    0$, with $h(\cdot)$ the solution of~\eqref{eq:probsomab=1}}.
\end{theorem}
\begin{proof}
    Without loss of generality, we admit that $x^k\not\in \X^*$. If $x^{p}\in
    \X^*$, for some $p\in\mathbb{N}$, \change{$\|h(x^k)\|=0$ for all $k\ge p$}.
    Suppose, by contradiction, that \change{$\|h(x^k)\| \not\rightarrow 0$}, then there
    exist $\mu>0$ and a subsequence $\mathbb{N}_1\subset \mathbb{N}$ such that
    \change{$$\|h(x^k)\|\ge\mu, \ k\in\mathbb{N}_1.$$ }

    Combining the last expression with  Lemma~\ref{lem:pre-converg1}, yields
    the following inequality:
    \begin{equation}\label{eq:lem3-eq1}
        \mathbb{E}\left[F(x^{k+1})| \B^k \right] - F(x^k)
        \le -\dfrac{\mu^2}{2m\delta_{DP}}, \  k\in\mathbb{N}_1.
    \end{equation}
    
    Taking the expectation on $\{\B^k\}_{k\in\mathbb{N}}$ in
    inequality~\eqref{eq:lem3-eq1}, since the random variables are i.i.d., we
    get
    \begin{equation}\label{eq:lem3-eq2}
        \mathbb{E}\left[F(x^{k+1})\right] - 
        \mathbb{E}\left[F(x^k) \right]
        \le -\dfrac{\mu^2}{2m\delta_{DP}}, \  k\in\mathbb{N}_1.
    \end{equation}
    
    From the expression~\eqref{eq:espvalue3} and
    Assumption~\ref{ass:lim-func}, we have that the sequence
    $\{\mathbb{E}\left[F(x^k)\right]\}_{k\in\mathbb{N}}$ is nonincreasing and
    bounded below, therefore, it converges. Its convergence along
    with~\eqref{eq:lem3-eq2} guarantees the contradiction, since the left side
    term of~\eqref{eq:lem3-eq2} goes to zero, while the right side term
    of~\eqref{eq:lem3-eq2} is negative.
\end{proof}

\change{Finally, we present a complexity result for
Algorithm~\ref{alg:ActivePCDM}, estimating \change{the expected objective
decrease in its sequences}}.

\begin{theorem}\label{teo:local-converg}
    Let $\{x^ k\}$ be a sequence generated by the \change{Parallel Active} BCDM
    Algorithm and suppose that the
    Assumptions~\ref{ass:lim-func},~\ref{ass:error-bound}
    and~\ref{ass:lips-contin} are verified. Then, there exists
    $k_1\in\mathbb{N}$ with the following linear convergence rate for the
    expected values of the objective function
    \[ 
        \mathbb{E}\left[F(x^{k})\right] - F^* \le  
        \left(1 - \dfrac{1}{m\delta_{DP}
        (1+(\overline{\beta} + L_f)\epsilon_2^2)}  \right)^{k-k_1} 
        (F(x^0) - F^*), 
    \]
    with $\overline{\beta}$ as in~\eqref{eq:betabarra}
    (Lemma~\ref{lem:pre-converg1}), $\epsilon_2$ as in
    Assumption~\ref{ass:error-bound} and $L_f$ as in
    Assumption~\ref{ass:lips-contin}.
\end{theorem}
\begin{proof}
    From \change{Theorem}~\ref{lem:pre-converg3} we have \change{ $\|h(x^k)\|\rightarrow 0$}.
    Hence, Assumption~\ref{ass:error-bound} ensures that, for a fixed $\varrho
    = F(x^0)$, there exist $\epsilon_2$>0 and $k_1\in\mathbb{N}$ such that
    \change{\begin{equation}\label{eq:teolocalconv-eq2}
        \|x^k - z^k\| \le \epsilon_2 \|h(x^k)\|, \ \forall \ 
        k \ge k_1,
    \end{equation}}
    in which $z^k\in\X^*$ satisfies $\|x^k - z^k\|=\text{dist}(x^k,\X^*)$.
    
    From Corollary~\ref{cor:quad-sup-bound}, we have
    \begin{equation}\label{eq:teolocalconv-eq4}
        f(x) + \nabla f(x)^T(y-x) \le f(y) + \dfrac{L_f}{2}\|y - x\|^2, 
        \ \forall \ y, x \in \R^n.
    \end{equation}
    
    Adding $\tfrac{\overline{\beta}}{2} \|x-y\|^2 +\psi(y)$ to both sides 
    of inequality~\eqref{eq:teolocalconv-eq4} yields
    \begin{eqnarray}\label{eq:teolocalconv-eq5}
        Q(y-x,x,\overline{\beta}) 
        & = & f(x) + \nabla f(x)^T(y-x) + \dfrac{\overline{\beta}}{2} 
        \|x-y\|^2 +\psi(y)\nonumber\\
        & \le & f(y) + \dfrac{\overline{\beta} + L_f}{2}\|x-y\|^2 
        +\psi(y)\nonumber\\ 
        & = & F(y) + \dfrac{\overline{\beta} + L_f}{2}\|x-y\|^2.
    \end{eqnarray}
    
    As a consequence of~\eqref{eq:teolocalconv-eq5}, we obtain
    \begin{equation}\label{eq:teolocalconv-eq6}
        Q(z^k-x^k,x^k,\overline{\beta}) 
        \le   F^* + 
        \dfrac{c}{2}\|x^k-z^k\|^2,\ c=\overline{\beta} + L_f.
    \end{equation}
    
    Putting together the expressions~\eqref{eq:lem1-eq2},
    ~\eqref{eq:teolocalconv-eq6}, and using that \change{$\min\limits_{s\in\X}
    Q(h,x^k,\overline{\beta})=
    Q(h^{\overline{\beta}}(x^k),x^k,\overline{\beta})$}, we have
    \begin{align}\label{eq:teolocalconv-eq7}
        \mathbb{E}\left[F(x^{k+1})| \B^k\right] 
        & \le  \left( 1 - \dfrac{1}{m\delta_{DP}} \right) F(x^k) + 
        \dfrac{1}{m\delta_{DP}} \left( F^* + 
        \dfrac{c}{2}\|x^k-z^k\|^2\right)\nonumber\\
        & =  F(x^k) - \dfrac{1}{m\delta_{DP}}( F(x^k) - F^*)  + 
        \dfrac{c}{2m\delta_{DP}} \|x^k-z^k\|^2,
    \end{align}
    for all $k\ge k_1$.
    
    Subtracting $F^*$ from both sides of 
    inequality~\eqref{eq:teolocalconv-eq7} and using the 
    expression~\eqref{eq:teolocalconv-eq2}, we obtain
    \change{\begin{eqnarray}\label{eq:teolocalconv-eq8}
        \mathbb{E}\left[F(x^{k+1})| \B^k\right] - F^*
        & \le & \left(1 - \dfrac{1}{m\delta_{DP}} \right)( F(x^k)
        - F^*) + \nonumber\\
        &  & + \dfrac{c\epsilon_2^2}{2m\delta_{DP}} \|h(x^k)\|^2,
    \end{eqnarray}}
    for all $k\ge k_1$.
    
    In  view of  Lemma~\ref{lem:pre-converg1}, the 
    relationship~\eqref{eq:teolocalconv-eq8} turns into
    \begin{eqnarray*}
        \mathbb{E}\left[F(x^{k+1})| \B^k\right] - F^*
        & \le & \left(1 - \dfrac{1}{m\delta_{DP}} \right)( F(x^k)
        - F^*) + \nonumber\\
        &   & + c\epsilon_2^2 
        (F(x^k) - \mathbb{E}\left[F(x^{k+1})| \B^k\right])\nonumber\\
        & = & \left(1 - \dfrac{1}{m\delta_{DP}} + 
        c\epsilon_2^2\right)
        ( F(x^k) - F^*) + \nonumber\\
        &   & + c\epsilon_2^2 
        (F^* - \mathbb{E}\left[F(x^{k+1})| \B^k\right]).
    \end{eqnarray*}
    This may be rewritten as
    \begin{equation}\label{eq:teolocalconv-eq9}
        \mathbb{E}\left[F(x^{k+1})| \B^k\right] - F^*
        \le \left(1 - \dfrac{1}{m\delta_{DP}
        (1+c\epsilon_2^2)} 
        \right)( F(x^k) - F^*),
    \end{equation}
    for all $k\ge k_1$.
    
    Taking both sides of~\eqref{eq:teolocalconv-eq9} to the 
    expectation conditioned to 
$\{\B^k\}_{k\in\mathbb{N}}$, 
    which are i.i.d. random variables, yields
    \begin{equation}\label{eq:teolocalconv-eq10}
        \mathbb{E}\left[F(x^{k+1}) \right] - F^*
        \le \left(1 - \dfrac{1}{m\delta_{DP}
        (1+c\epsilon_2^2)} 
        \right)(\mathbb{E}\left[ F(x^k) \right] - F^*),
    \end{equation}
    for all $k\ge k_1$.
    
    The convergence result follows by applying the expression~\eqref{eq:teolocalconv-eq10} 
    repeatedly, since by~\eqref{eq:espvalue3}, the sequence 
$\{\mathbb{E}\left[ F(x^k)\right]\}$ is nonincreasing.
\end{proof}

\change{
    The complexity result above includes the term $k - k_1$ in the exponent of the linear rate, meaning that linear convergence is only ensured from iteration $k_1$ on. This is needed to accommodate the nature of Assumption~\ref{ass:error-bound}. This assumption only asserts the error bound when $F$ and $h$ are small enough, that we assume will happen at iteration $k_1$ in the proof of Theorem~\ref{teo:local-converg}. However, in some special cases it is possible to estimate $k_1$. A notable example is when $f$ is strongly convex with Lipschitz continuous gradients. In this case, \cite[Theorem 4]{TS09} shows that Assumption~~\ref{ass:error-bound} holds with $\epsilon = \infty$ and independently of $\rho$ for all $x \in \X$. Hence, in this case $k_1 = 0$, and Theorem~\ref{teo:local-converg} is a global linear convergence result in expectation. For other cases, estimating $k_1$ may be difficult, or even impossible, and problem dependent.
}

 \selectlanguage{english}

\section{Numerical tests}
\label{sec:numerical}

\change{This section presents} the numerical behavior of the parallel version of
Active BCDM\change{,} Algorithm~\ref{alg:ActivePCDM},
\change{denoted \texttt{PA} from now on}. In~\cite{LSS19}, the
authors introduced its serial variant. \change{It showed} that active
constraints \change{identification} combined with a nonuniform choice of
coordinate blocks \change{was} very efficient and competitive \change{with}
several well-established methods in the literature. Therefore, in this work,
we will focus on how parallelism can accelerate \change{this} method and compare
\change{the effectiveness of the acceleration achieved by} \texttt{PA} with \change{the one achieved by} its uniform counterpart.

\change{ To accomplish this, one might consider it enough to compare the
behaviors of the new \texttt{PA} code and \texttt{PCDM}, the standard
implementation of parallel uniform block selection from~\cite{RT2016}.
However, as \texttt{PA} evolved from the code in~\cite{LSS19}, it is
implemented in Fortran 90, while \texttt{PCDM} is implemented in C++. This
difference, and the fact that such codes were developed by different groups,
could be partially responsible for the perceived differences in performance.
We introduced our implementation of the uniform selection method in Fortran
90, \texttt{UBCDM}, to mitigate this. Note that Fortran 90 is considered to
produce slightly faster executables due to its programming model that allows
the compiler to optimize the generated binary better. This should be
considered when interpreting the results below. We also point out that we
needed to perform minimal changes to \texttt{PCDM} to enable repeated
executions for each test instance so that the random selection of the blocks
could be taken into account.}

All experiments were performed on a Ryzen $9 7950X3D$ $16$C/$32$T system with
$128$ GB of memory, running Ubuntu Linux $22.04.4$ LTS. \change{They were
compiled using version 11.4.0 of the GNU compilers, and the parallelization
is achieved using OpenMP.} Since the CPU has $16$ processing cores with
independent floating point units, the experiments were limited to \change{use} a
maximum of 16 threads \change{always pinned to run in different cores}. For each
\change{test case}, the \change{codes} were executed with the number of threads
$\tau$ within the set $\{1, 2, 4, 8, 16\}$.

Throughout the section, the class of problems tested is the well-known
Lasso~\cite{TB96}, given by
\begin{equation}\label{eq:lasso}
    \min_x \dfrac{1}{2} \|Ax-b\|^2_2 + \lambda \|x\|_1, \quad \lambda>0, 
\end{equation}
 which can be reformulated as a problem with simple constraints
\begin{align}\label{eq:lassoref}
\min_{x_+,x_-} &\quad \frac{1}{2} \|A(x_+-x_-) - b\|^2+\lambda e^T(x_++x_-) \nonumber\\
\text{s.t.}    &\quad x_+\ge 0 \\
            &\quad x_-\ge 0, \nonumber
\end{align}
where $e=(1, \dots, 1)^T$. Note that all solutions $\overline{x}$ of~\eqref{eq:lasso} can be written in terms of solutions $(\overline{x}_+, \overline{x}_-)$ of~\eqref{eq:lassoref}, using $\overline{x} = \overline{x}_+ - \overline{x}_-$.

\change{Adopting the same choices of~\cite{RT2016}, the initial point
$x^0$ was set as the null vector; the initial cycle size $c_0$ was the problem
dimension $n$, and $\ell_{\max} = 1000 n$. Moreover, for updating the set
${\cal J}$, we have used the strategy described in~\cite[Section~6.1]{RT2016}.
}

\change{All algorithms evaluated in this study are variants of
coordinate descent methods that operate on coordinate blocks of size one,
i.e., individual variables. To account for the stochasticity inherent in the
coordinate selection process, each algorithm was executed multiple times on
the same problem instance. Specifically, 20 independent runs were
performed for the test cases described in Subsection~\ref{subs:perfcen}, and
100 runs for those in Subsection~\ref{subs:realcen}, provided
that the cumulative execution time to solve the instance exceeded one second. Otherwise, a sufficient number of runs were performed to complete one second of cumulative time.}

\change{Each run was stopped once the objective value of an iteration reached a
predefined target, which closely approximates the known optimal value. More
details on how such a target was obtained are presented in the sections describing
each test scenario below. Finally, the primary performance metric used in this
study is the average running time \change{of} all independent runs.}

Analogously to what was done in~\cite{RT2016}, even though the theory of
Parallel Active BCDM was originally formulated for the synchronous case, the
coordinate blocks are updated asynchronously. The updates of the blocks (of
size $1$) have a very low cost, compared to the effort necessary to
synchronize the gradient update of the smooth part of~\eqref{eq:lasso} after
the calculation of the descent directions. This makes synchronous
implementation impractical.

\change{In the serial case, the gradient can be efficiently updated at a low computational cost immediately following the modification of a single coordinate. However, in the asynchronous setting, multiple threads concurrently read from and write to the shared gradient vector, leading to a "race condition" that can cause the gradient to accumulate numerical errors over time. This gradient degradation can severely impact the algorithm's convergence.}

\change{To mitigate this issue, our implementation includes a
correction routine that runs after each cycle that comprises $n$ coordinate
updates. This routine recomputes and
corrects the shared gradient vector to restore numerical consistency.
Additionally, this synchronization point is leveraged within the framework of
Algorithm~\ref{alg:ActivePCDM} to identify the active constraints, thus
improving the algorithmic performance.}

\subsection{A controlled and highly favorable scenario}\label{subs:perfcen}

The initial experiments \change{were} performed in a controlled environment that
is particularly favorable to coordinate descent methods. As theoretical
results suggest, the smaller the degree of partial separability $\omega$, the
greater the expected acceleration for these methods. For
problem~\eqref{eq:lasso}, the value of $\omega$ depends on the number of
nonzero elements in the rows of the matrix $A$.

To generate the initial test instances, we employed the random problem
generator of~\cite{RT2016}, which the authors provide in a C++ implementation.
\change{This tool builds a class of test problems originally described
in~\cite{NE13}. Using this generator, it is possible to construct a matrix $A$
and a vector $b$ for the problem~\eqref{eq:lasso}, assuming $\lambda=1$, based
on the following input parameters: the dimensions of matrix $A$; the number of
nonzero entries per row in matrix $A$, and the sparsity level of the optimal
solution, that is, the number of its nonzero components. By default, the
\change{sparsity level} is set to $\min\{10000,\tfrac{m}{2}\}$ and we used
this value.} Consequently, the value of \change{the degree of
partial separability} $\omega$ \change{is} indirectly influenced by these
inputs. 

\change{We generated six random problem instances (matrix $A$ and
vector $b$). Among these, three instances were constructed with $20$ nonzero
elements per column, while the remaining three contained $1000$ nonzero
elements per column. Table~\ref{tab:data-ne} summarizes the set of
generated problems, including their identifiers, the dimensions of matrix $A$,
the corresponding values of $\omega$, and the proportion of
zero components in the optimal generated solution $x^*$,
denoted by ($nz(x^*)$).}

\change{
\arrayrulecolor{red}
\begin{table}[]
    \centering
    \begin{tabular}{c|c|c|c} \hline
        Name & \#rows $\times$ \#cols & $\omega$ & $nz(x^*)$  \\ \hline
        $Ne1$ &  $100{,}000\times 200{,}000$ & 258 & $95.0\%$  \\ \hline
        $Ne2$ &  $200{,}000\times 100{,}000$ & 85 & $90.0\%$  \\ \hline
        $Ne3$ &  $2{,}000{,}000\times 1{,}000{,}000$ & 90 & $99.0\%$ \\ \hline
        $Ne4$ &  $100{,}000\times 200{,}000$ & 71 & $95.0\%$  \\ \hline
        $Ne5$ &  $200{,}000\times 100{,}000$ & 27 & $90.0\%$ \\ \hline
        $Ne6$ &  $2{,}000{,}000\times 1{,}000{,}000$  & 30 & $99.0\%$ \\ \hline
    \end{tabular}
    \caption{Features of the artificially generated problems.}
    \label{tab:data-ne}
\end{table}
\arrayrulecolor{black} 
}

To define the target function value ($F_{target}$) for each of the six test
problems, we employed our implementation of the \change{uniform coordinate descent method, running for $1000n$ iterations with a block size equal to $1$},
\change{obtaining a very precise approximation of the optimal value that we
denote $\tilde{F}^*$.} Finally, we constructed the target value so that the
relative error between $F_{target}$ and $\tilde{F}^*$ is $10^{-4}$, that is,
$F_{target}$ is defined as
\[
F_{target} = \tilde{F}^*(1+10^{-4}).
\]

\change{We begin analyzing the performance of
Algorithm~\ref{alg:ActivePCDM} to identify an effective configuration. This
process requires the appropriate selection of two key parameters:
$\delta_{DP}$  and $\delta_{F}$. The first governs the
probability distribution employed by the method, that is, the relative
frequency with which the method selects inactive coordinates compared to
active ones. The second one determines the number of iterations between
successive evaluations of the identification function. In this preliminary
analysis, we fix $\delta_{F}=1$ and evaluate the algorithm's performance with
different values of $\delta_{DP}\in\{10,50,100,500,1000\}$.  In the
experimental results shown in the figures, each variant of the algorithm is
labeled as \texttt{PA} followed by the corresponding value of $\delta_{DP}$
and the number of threads used. To put the different
configurations in perspective, we have used $\log_2$ scaled performance
profiles of the average execution time among the 20 independent runs of each
problem instance solved by each variant of the algorithm. Such cumulative
distribution plots depict the proportion of problems solved in the $y$-axis,
within the factor of the fastest ($\log_2$ scaled) that is shown in the
$x$-axis. We refer the reader to~\cite{DolanMore2002} for further details on
this benchmarking tool.}

\begin{figure}
    \begin{subfigure}[h]{0.5\linewidth}
        \includegraphics[width=\linewidth]{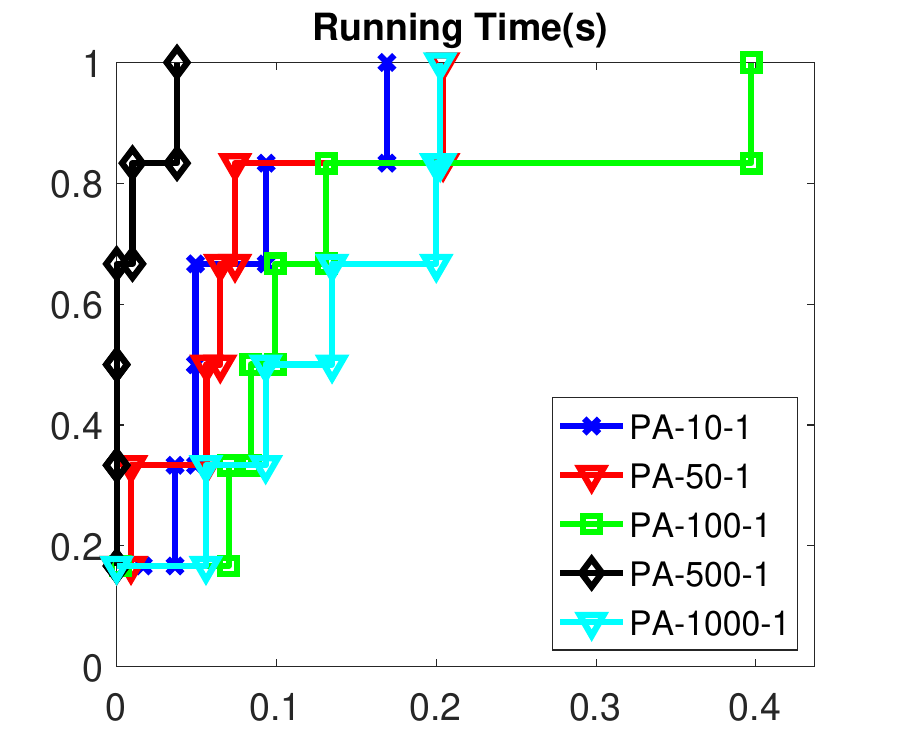}
        \caption{$1$ thread.}
    \end{subfigure}
    \hfill
    \begin{subfigure}[h]{0.5\linewidth}
        \includegraphics[width=\linewidth]{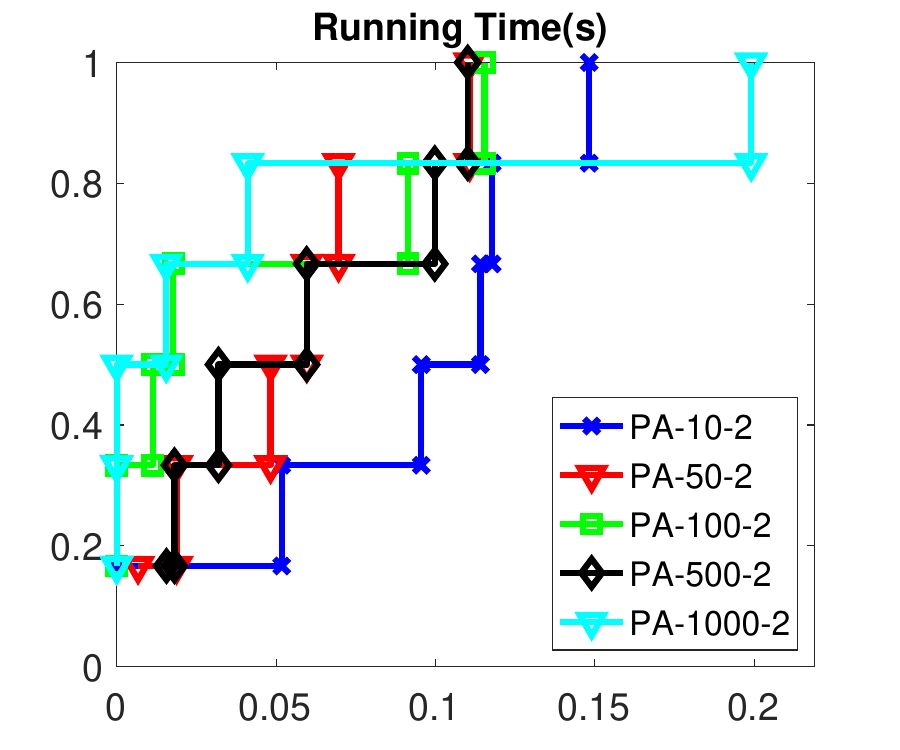}
        \caption{$2$ threads.}
    \end{subfigure}\caption{Performance profile of the average execution time between the variants of the Parallel Active BCDM method (\texttt{PA}) with $1$ and
    $2$ threads.}
    \label{fig:fig4}
\end{figure}

\begin{figure}
    \begin{subfigure}[h]{0.33\linewidth}
        \includegraphics[width=\linewidth]{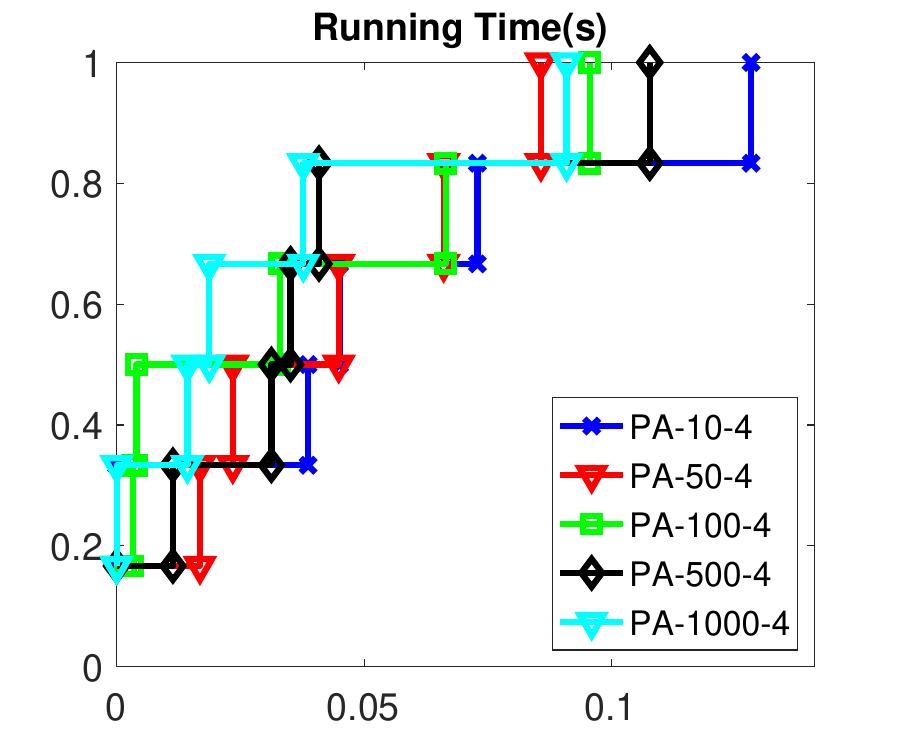}
        \caption{$4$ threads.}
    \end{subfigure}
    \hfill
    \begin{subfigure}[h]{0.33\linewidth}
        \includegraphics[width=\linewidth]{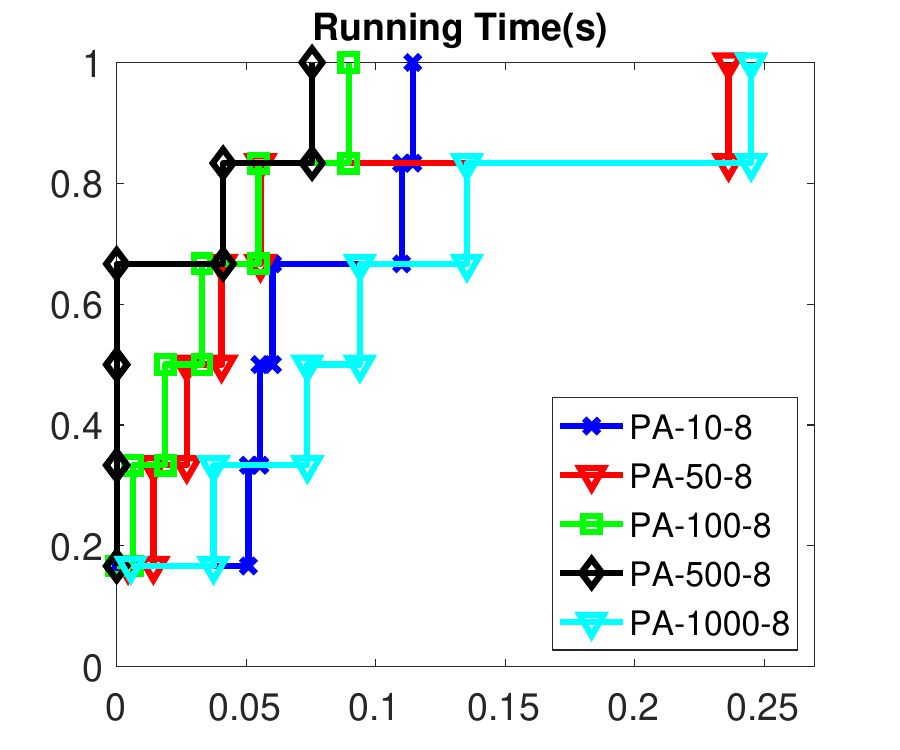}
        \caption{$8$ threads.}
    \end{subfigure}\begin{subfigure}[h]{0.33\linewidth}
        \includegraphics[width=\linewidth]{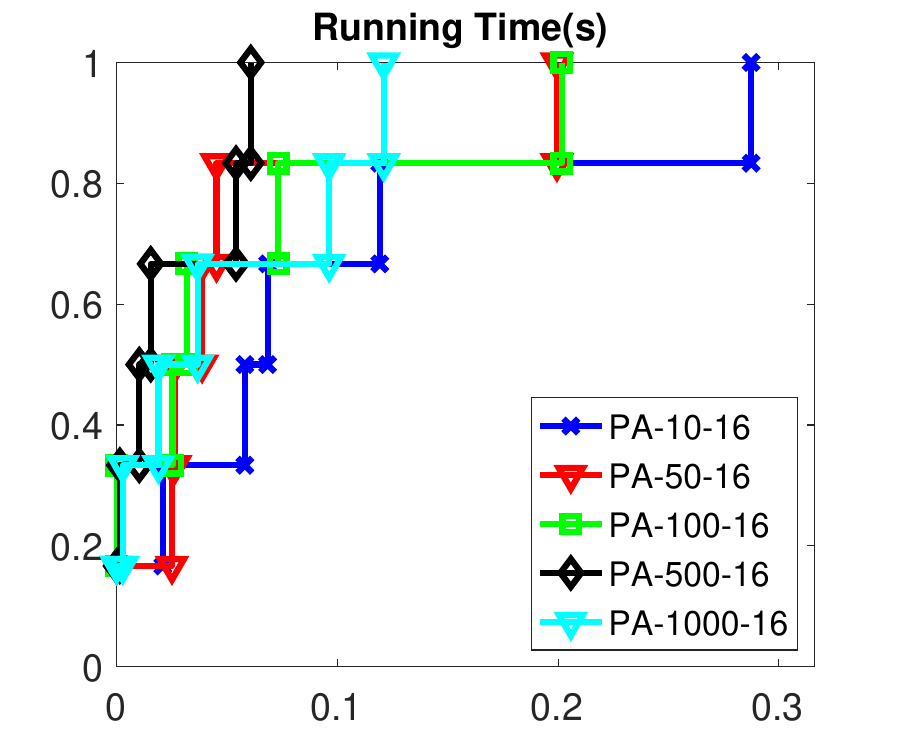}
        \caption{$16$ threads.}
    \end{subfigure}\caption{Performance profile of the average execution time between the variants of the Parallel Active BCDM method (\texttt{PA}) with $4$,
    $8$, and $16$ threads.}
    \label{fig:fig5}
\end{figure}

\change{Looking at Figures~\ref{fig:fig4} 
and~\ref{fig:fig5}, we note that the performance of the 
methods is minimally affected by the change in the value 
of $\delta_{DP}$ for this set of problems. Due to 
the promising performance of the \texttt{PA-500} variant, this choice 
was adopted for the subsequent tests within this subsection.}

\change{Using the selected \texttt{PA} variant, we evaluate the
multi-threaded speedup in execution time for the \texttt{UBCDM},
\texttt{PCDM}, and \texttt{PA-500} methods for the six randomly generated $Ne$
test problems. To illustrate the performance of these methods, we present the
results in the form of boxplots, as shown in Figure~\ref{fig:rat-ne}. Each
boxplot represents the speedup ratio, defined as the average execution time of
the serial version of a given method divided by the average execution time of
its multi-threaded counterpart. This visualization provides a clear comparison
of the parallel efficiency achieved by the different algorithmic
implementations. 
}

\begin{figure}
    \begin{subfigure}[h]{0.33\linewidth}
        \includegraphics[width=\linewidth]{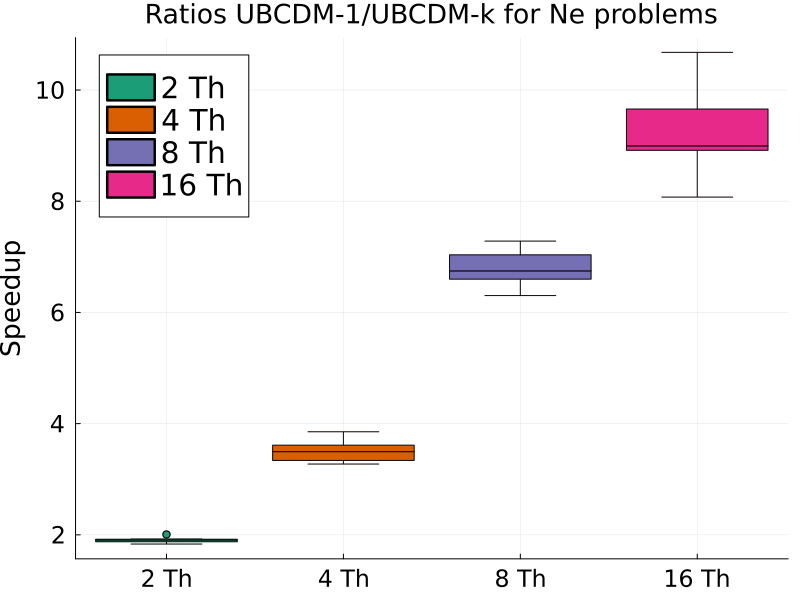}
\end{subfigure}
    \hfill
    \begin{subfigure}[h]{0.33\linewidth}
        \includegraphics[width=\linewidth]{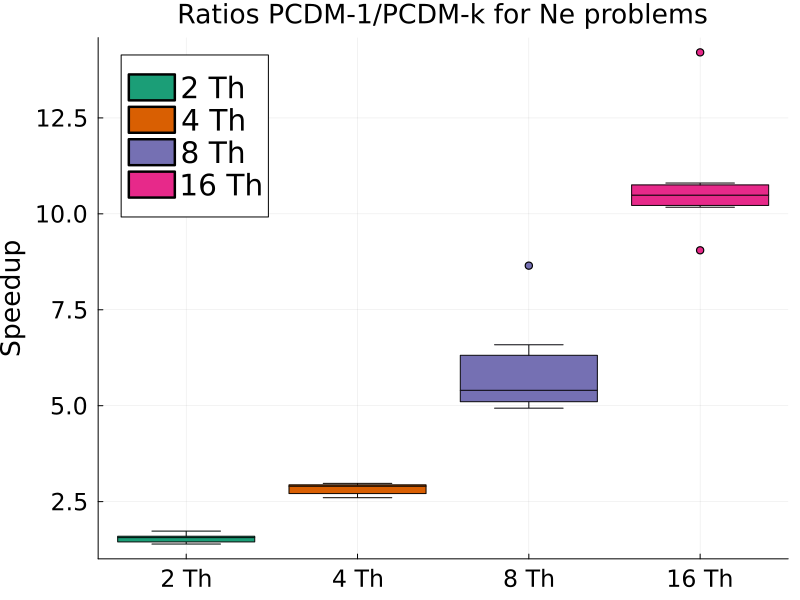}
\end{subfigure}\begin{subfigure}[h]{0.33\linewidth}
        \includegraphics[width=\linewidth]{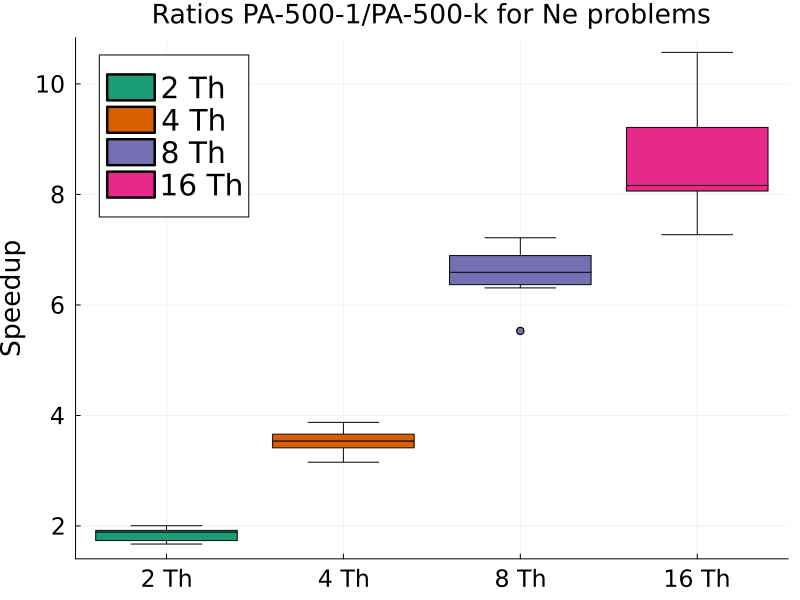}
\end{subfigure}\caption{Boxplots comparing the speedup in running time between
    \texttt{UBCDM-}$1$/\texttt{UBCDM-}$k$ (left),
    \texttt{PCDM-}$1$/\texttt{PCDM-}$k$ (center) and
    \texttt{PA-500-}$1$/\texttt{PA-500-}$k$ (right), with $k\in \{ 2, 4, 8,
    16\}$ for $Ne$ problems.}
    \label{fig:rat-ne}
\end{figure}

\change{Figure~\ref{fig:rat-ne} shows that all methods scale
effectively in terms of speedup for the problems considered: as the number of
threads increases, the observed speedup improves consistently. The
\texttt{UBCDM} and \texttt{PA-500} methods exhibit similar speedup behavior
for the tested configurations. In contrast, the \texttt{PCDM} method achieves
a comparable average speedup when using $2$, $4$, and $8$ threads but surpasses
the other approaches in average speedup when executed with $16$ threads.}

\begin{figure}[h]
    \begin{subfigure}[h]{0.50\linewidth}
        \includegraphics[width=\linewidth]{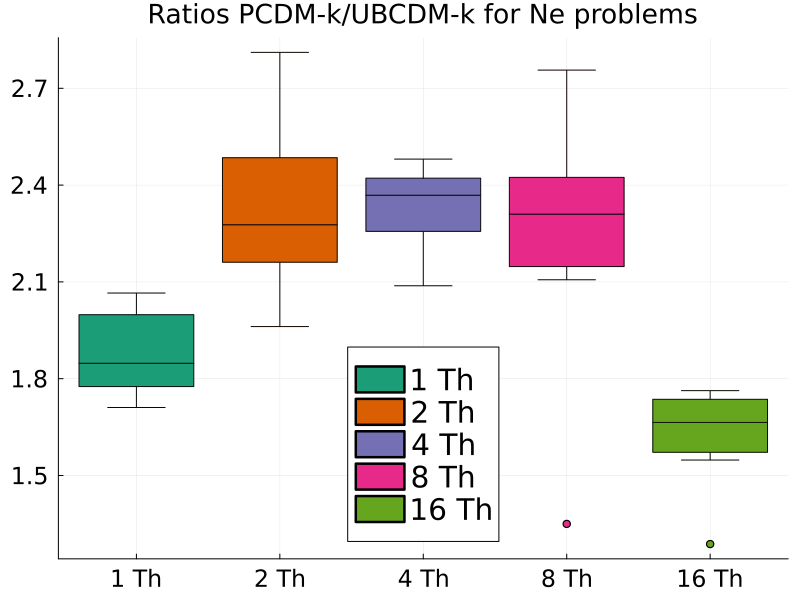}
\end{subfigure}\begin{subfigure}[h]{0.50\linewidth}
    \includegraphics[width=\linewidth]{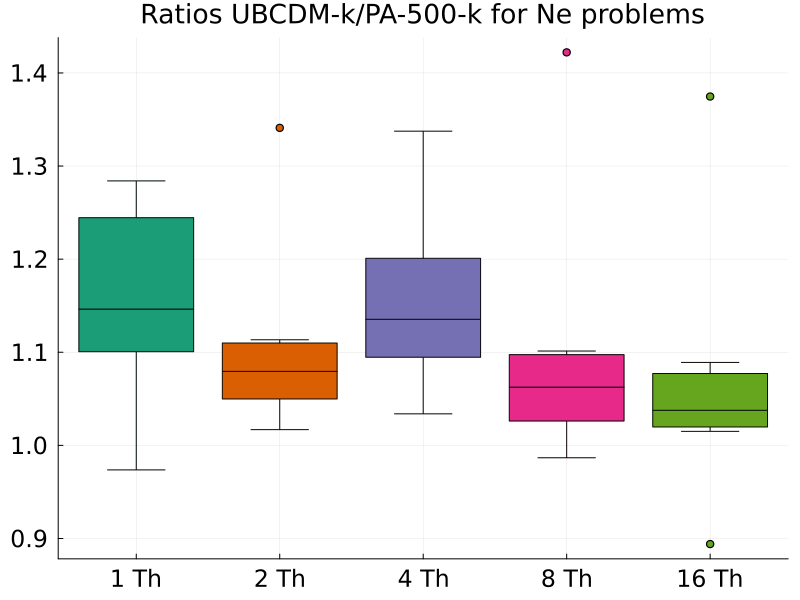}
\end{subfigure}\caption{Boxplot comparing the running time ratio between \texttt{PCDM} and \texttt{UBCDM} (left) and between \texttt{UBCDM} and \texttt{PA-500} (right) for $Ne$ problems, with $1$, $2$, $4$, $8$ and $16$ threads.}
     \label{fig:rat-pbub-ne1}
\end{figure}

\change{We conclude this subsection by presenting two boxplots in
Figure~\ref{fig:rat-pbub-ne1}, which compare the execution time ratios between
the \texttt{PCDM} and \texttt{UBCDM}  and between the \texttt{UBCDM} and
\texttt{PA-500} for different thread values. The results indicate that
\texttt{UBCDM} achieves, on average, a two-times speedup relative to
\texttt{PCDM}. In turn, for this set of problems,
\texttt{PA-500} demonstrates a modest performance advantage over
\texttt{UBCDM}, with an average speedup slightly greater than one.
Additionally, in Section~\ref{subs:ne} of the appendix,
we present tables reporting the execution time and the average
number of iterations for the most relevant methods discussed in this
subsection, namely \texttt{UBCDM}, \texttt{PCDM}, and \texttt{PA-500}.
The values displayed in such tables contrasting the
pairs $(Ne1,Ne4)$, $(Ne2,Ne5)$, and $(Ne3,Ne6)$ corroborate that the smaller
the degree of partial separability $\omega$, the greater the expected
acceleration for the coordinate descent methods, as suggested by the
theoretical results.}

\change{It is important to caution the reader that, despite the
favorable speedup presented in the plots of Figure~\ref{fig:rat-ne},
consistent with the principle discussed in~\cite{TTR2018}, it should be
interpreted as indicative of the best-case performance achievable by parallel
coordinate descent methods. In real-world applications, problem instances
often involve matrices~$A$ with less favorable structures, particularly with
columns containing highly nonuniform quantities of nonzero elements. The
structural uniformity present in the test cases used here may bias the
results, obscuring several challenges that coordinate descent methods face
when applied to more heterogeneous and irregular matrix structures. 
The experiments in the next subsection were performed
to provide further insight into this matter.} 

\subsection{A realistic scenario}\label{subs:realcen}

\change{In this subsection, we evaluate the performance of block
coordinate descent methods in real-world problems. For this purpose, we use
the set of $49$ test instances introduced in~\cite[Tables 2 and 3]{LSS19}.
These tables provide detailed information on each problem, including its name,
source, matrix dimensions, target objective function value, and the number of
zero coordinates in the optimal solution. The regularization parameter is
chosen as $\lambda=0.1\|A^Tb\|_{\infty}$, where $A$ and $b$ denote the matrix
and vector that define each problem, respectively. Because these problems are
derived from real applications and contain actual data, they present great
structural diversity, making them particularly valuable for analysis.}

\change{A limitation of this test set is that most problems have a
relatively small number of samples and/or variables compared to those
analyzed in the previous subsection. Another distinguishing characteristic is
the presence of two problem classes: those in which the matrix $A$ has more
columns than rows, called $SC$ problems, and those in which there are more
rows than columns, called $SR$ problems. In the context of Lasso, where the
primary objective is the selection of variables or characteristics, this
distinction is particularly relevant. In the $SC$ case, the problem includes
many potentially redundant or insignificant features to be eliminated. In
contrast, the $SR$ case involves fewer variables from the outset, where a more
refined selection is required. These differences were already important
in~\cite{LSS19}, which reported that the behavior of the block coordinate
descent method with identification differs notably between these two problem
classes. Therefore, we present the results for each class separately to
highlight the differences, which also impact the performance of the
parallel implementations.}

\change{We begin by analyzing the $SC$-type problems. The parameter
$\delta_{DP}$ is calibrated using the same range of values as considered in
Subsection~\ref{subs:ne}. The comparison is once more conducted
using performance profiles now based on the
average execution time among 100 independent runs of each problem
instance for all \texttt{PA} method variants, on the set of $25$ $SC$-type
problems. These results are depicted in Figures~\ref{fig:pa-sc-all1}
and~\ref{fig:pa-sc-all2}. From the performance profiles, it is evident that
the \texttt{PA-10} variant achieves the best overall performance among the
tested configurations.}

\begin{figure}
    \begin{subfigure}[h]{0.5\linewidth}
        \includegraphics[width=\linewidth]{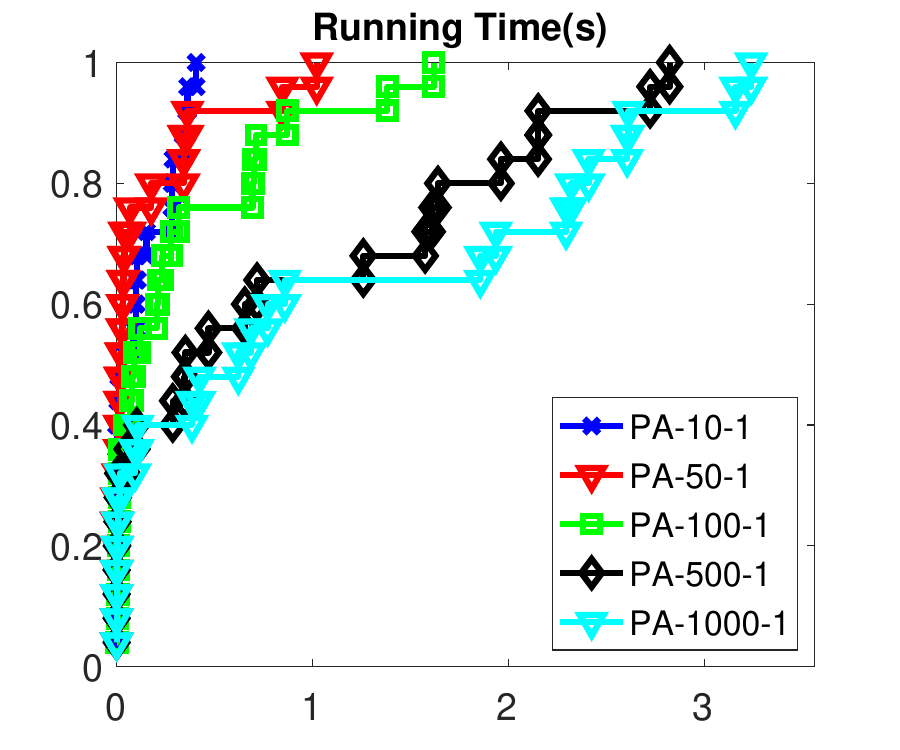}
        \caption{$1$ thread.}
    \end{subfigure}
    \hfill
    \begin{subfigure}[h]{0.5\linewidth}
        \includegraphics[width=\linewidth]{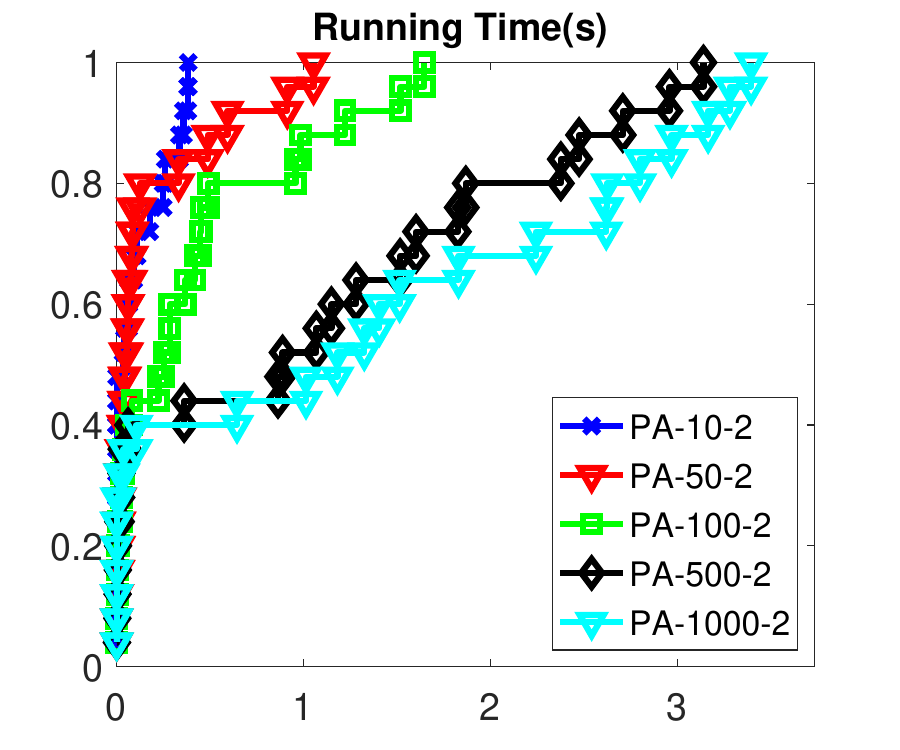}
        \caption{$2$ threads.}
    \end{subfigure}\caption{Performance profile of the average execution time between the variants of the Parallel Active BCDM method
    (\texttt{PA}) with $1$ and $2$ threads for the $25$ $SC$ problems.}
    \label{fig:pa-sc-all1}
\end{figure}

\begin{figure}
    \begin{subfigure}[h]{0.33\linewidth}
        \includegraphics[width=\linewidth]{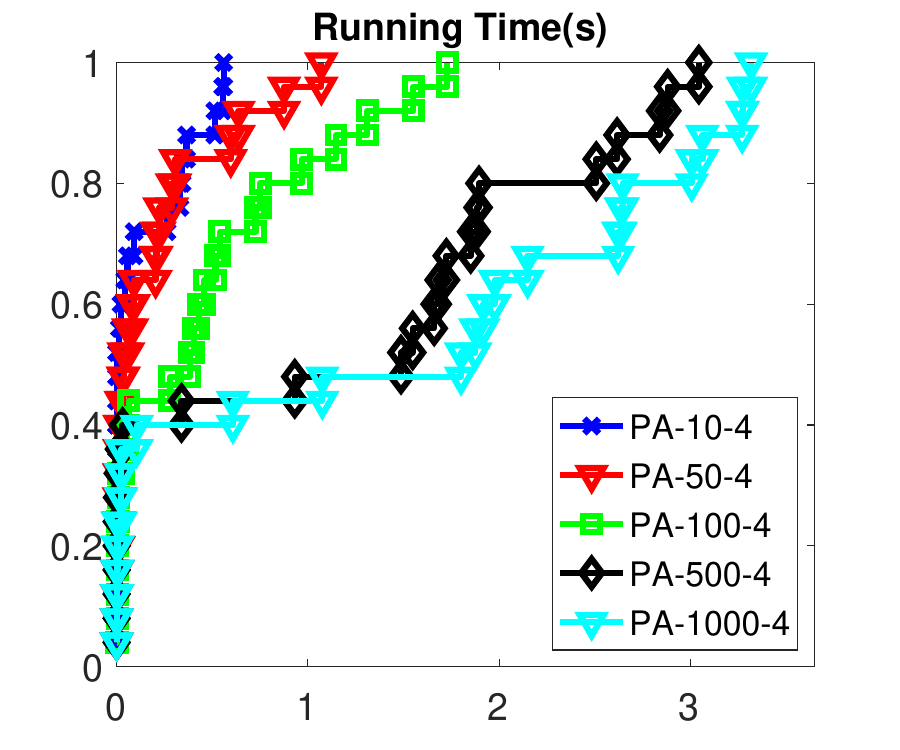}
        \caption{$4$ threads.}
    \end{subfigure}
    \hfill
    \begin{subfigure}[h]{0.33\linewidth}
        \includegraphics[width=\linewidth]{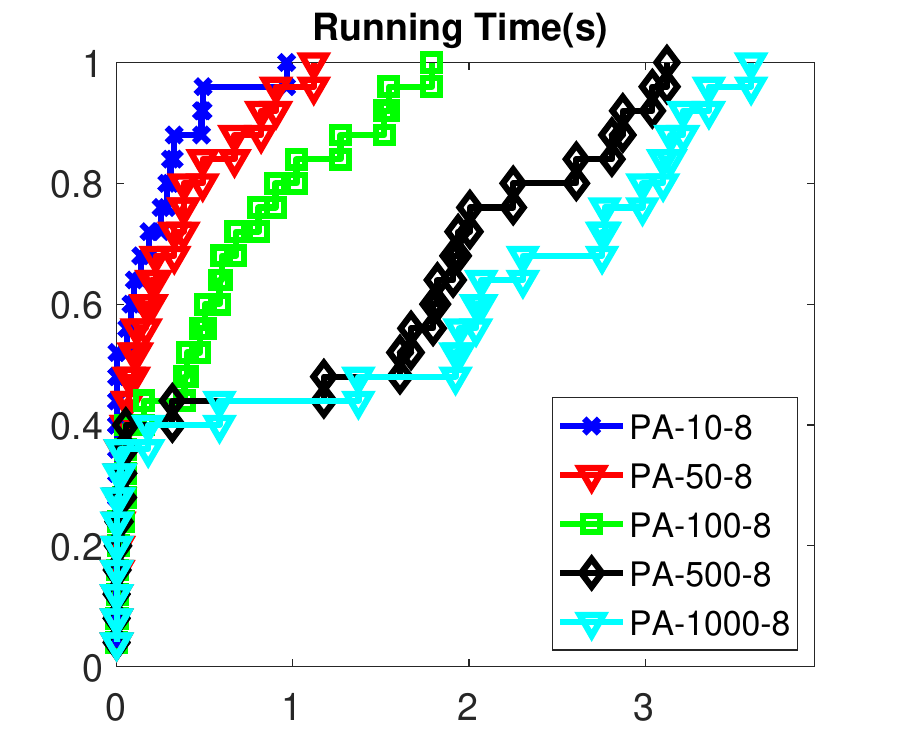}
        \caption{$8$ threads.}
    \end{subfigure}\begin{subfigure}[h]{0.33\linewidth}
        \includegraphics[width=\linewidth]{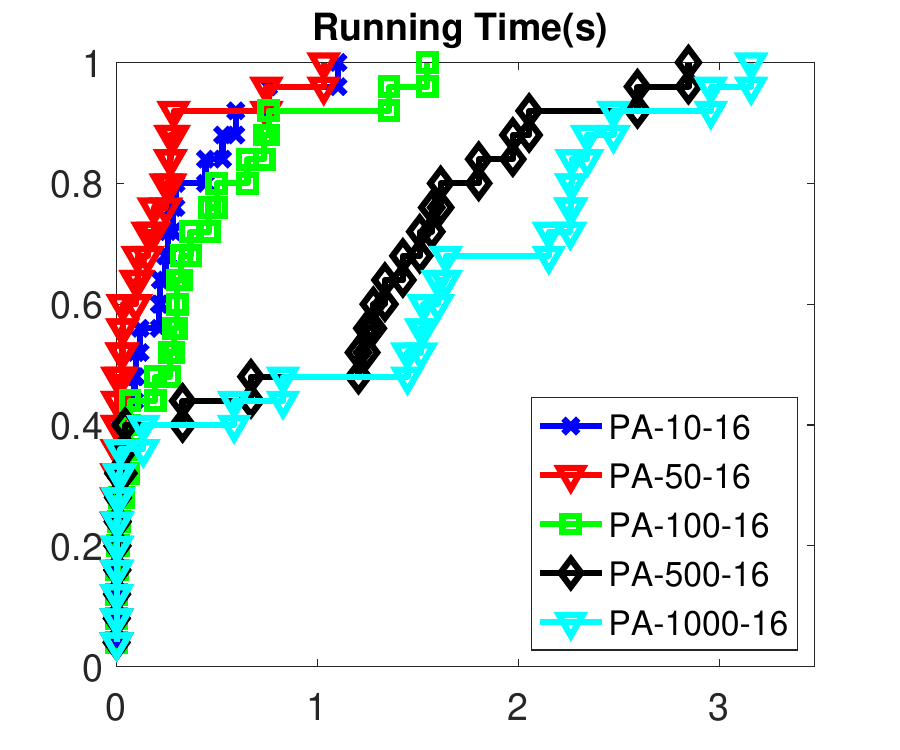}
        \caption{$16$ threads.}
    \end{subfigure}\caption{Performance profile of the average execution time between the variants of the Parallel Active BCDM method
    (\texttt{PA}) with $4$, $8$, and $16$ threads for the $25$ $SC$
    problems.}
    \label{fig:pa-sc-all2}
\end{figure}

\change{Figure~\ref{fig:rat-sc} presents three boxplots, one for each
method \texttt{PCDM}, \texttt{UBCDM} and \texttt{PA-10}, illustrating
the execution time speedup achieved by each method relative to its
corresponding single-threaded variant. Although the observed speedups are less
pronounced than those reported for the $Ne$ problems, the
performance of all methods improves consistently with increasing
number of threads. Notably, the \texttt{PCDM} method exhibits relatively
modest gains when using $2$ and $4$ threads. With $8$ and $16$ threads, its
speedup becomes more significant and surpasses the performance of the other
two methods under these configurations.}

\begin{figure}
    \begin{subfigure}[h]{0.33\linewidth}
        \includegraphics[width=\linewidth]{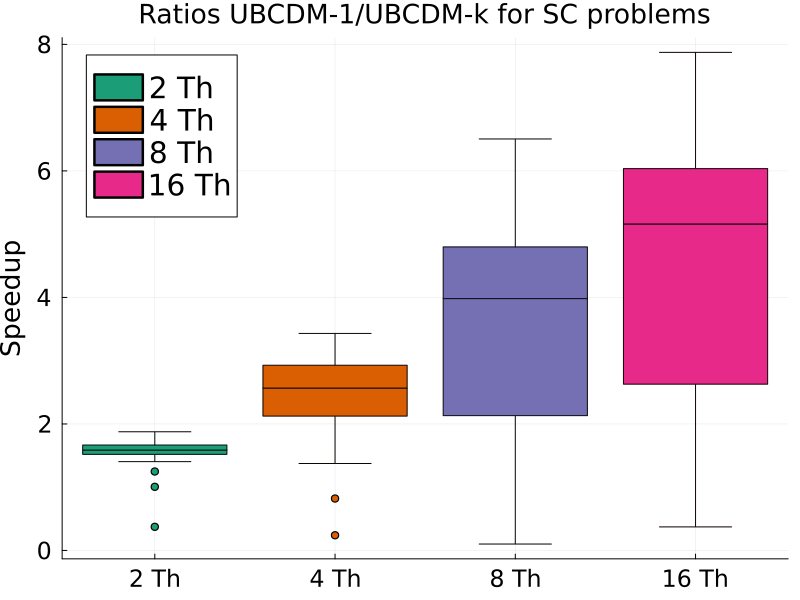}
\end{subfigure}
    \hfill
    \begin{subfigure}[h]{0.33\linewidth}
        \includegraphics[width=\linewidth]{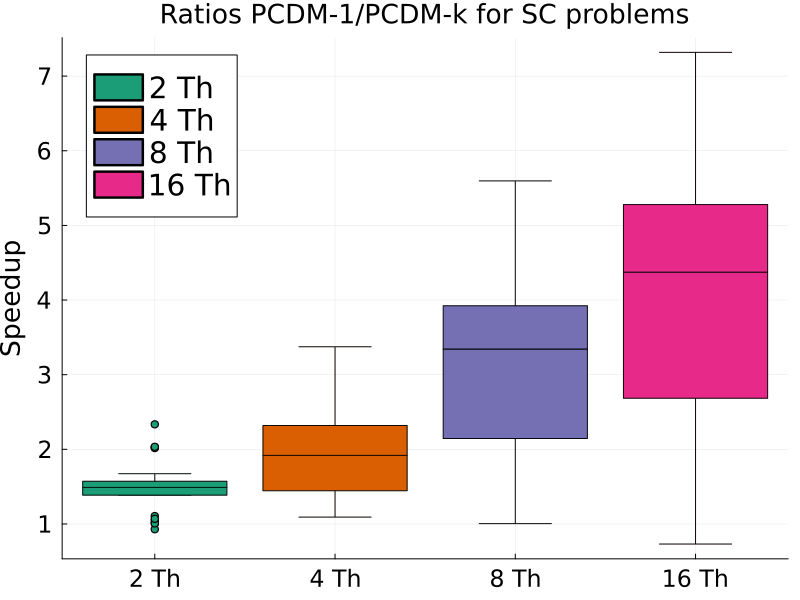}
\end{subfigure}\begin{subfigure}[h]{0.33\linewidth}
        \includegraphics[width=\linewidth]{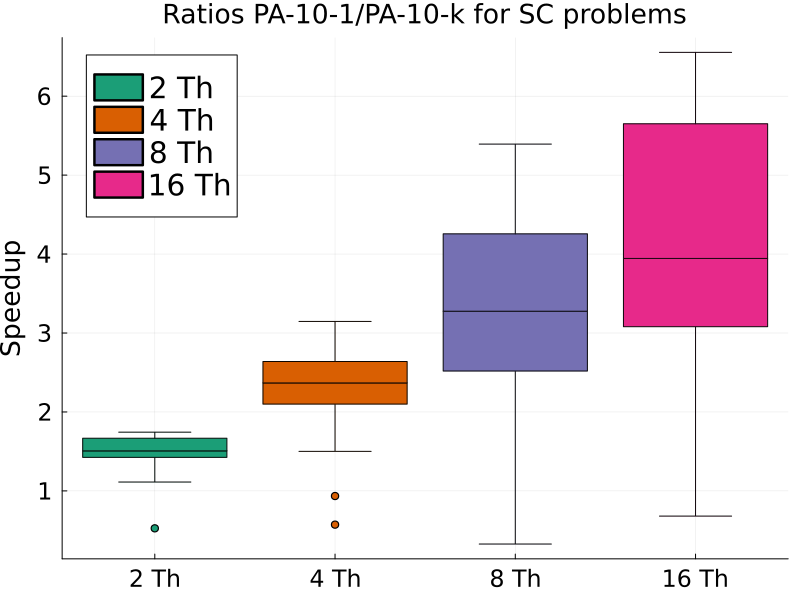}
\end{subfigure}\caption{Boxplots comparing the speedup in running time between
    \texttt{UBCDM-}$1$/\texttt{UBCDM-}$k$ (left),
    \texttt{PCDM-}$1$/\texttt{PCDM-}$k$ (center) and
    \texttt{PA-10-}$1$/\texttt{PA-10-}$k$ (right), with $k\in \{ 2, 4, 8,
    16\}$ for $SC$ problems.}
    \label{fig:rat-sc}
\end{figure}

\change{We conclude the comparison of the $SC$  problems with the
boxplots shown in Figure~\ref{fig:rat-pbub-sc1}, which depict the execution
time performance for the evaluated methods. Specifically, the figure contrasts
the performance of \texttt{PCDM} against \texttt{UBCDM}, as well as
\texttt{UBCDM} against \texttt{PA-10}. Again, \texttt{UBCDM} method achieves
an average speedup of approximately $2$ relative to \texttt{PCDM}. In turn,
\texttt{PA-10} attains an average speedup of nearly $3.4$ when compared to
\texttt{UBCDM}.}

\begin{figure}[h]
    \begin{subfigure}[h]{0.50\linewidth}
        \includegraphics[width=\linewidth]{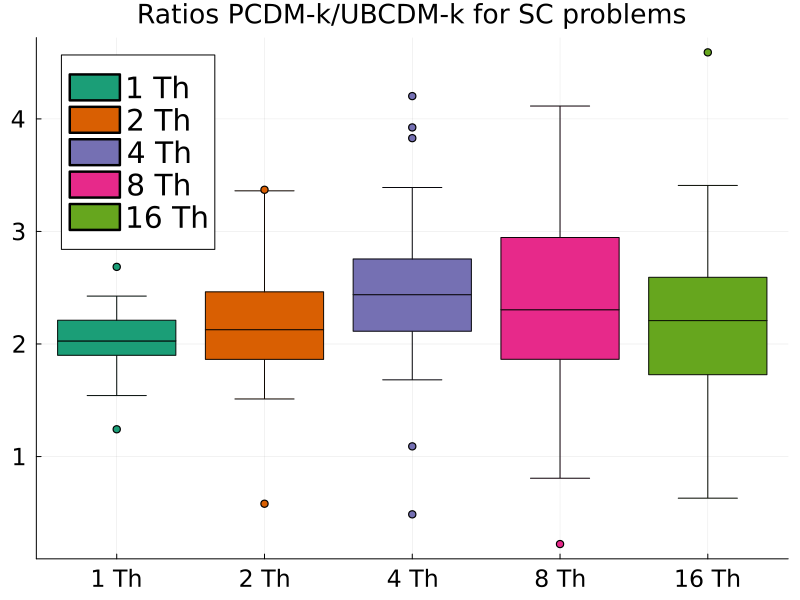}
\end{subfigure}\begin{subfigure}[h]{0.50\linewidth}
    \includegraphics[width=\linewidth]{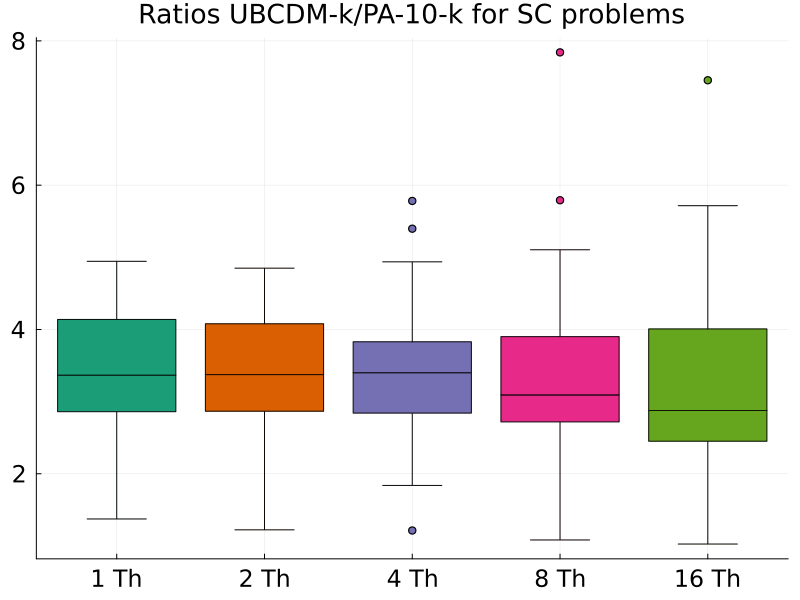}
\end{subfigure}\caption{Boxplot comparing the running time ratio between \texttt{PCDM} and \texttt{UBCDM} (left) and between \texttt{UBCDM} and \texttt{PA-10} (right) for $SC$ problems, with $1$, $2$, $4$, $8$ and $16$ threads.}
     \label{fig:rat-pbub-sc1}
\end{figure}

\change{We now turn to the analysis of the experimental results for
the problems of $SR$-type. The parameter $\delta_{DP}$ is adjusted
using the same range of values employed previously. Performance profiles are
used to compare the average execution times of all \texttt{PA} method variants
for the $24$ $SR$-type problems, as illustrated in
Figures~\ref{fig:pa-sr-all1} and~\ref{fig:pa-sr-all2}. Similarly to the $SC$
case, the most promising configuration corresponds to $\delta_{DP}=10$, which
delivers the best overall performance.}

\begin{figure}
    \begin{subfigure}[h]{0.5\linewidth}
        \includegraphics[width=\linewidth]{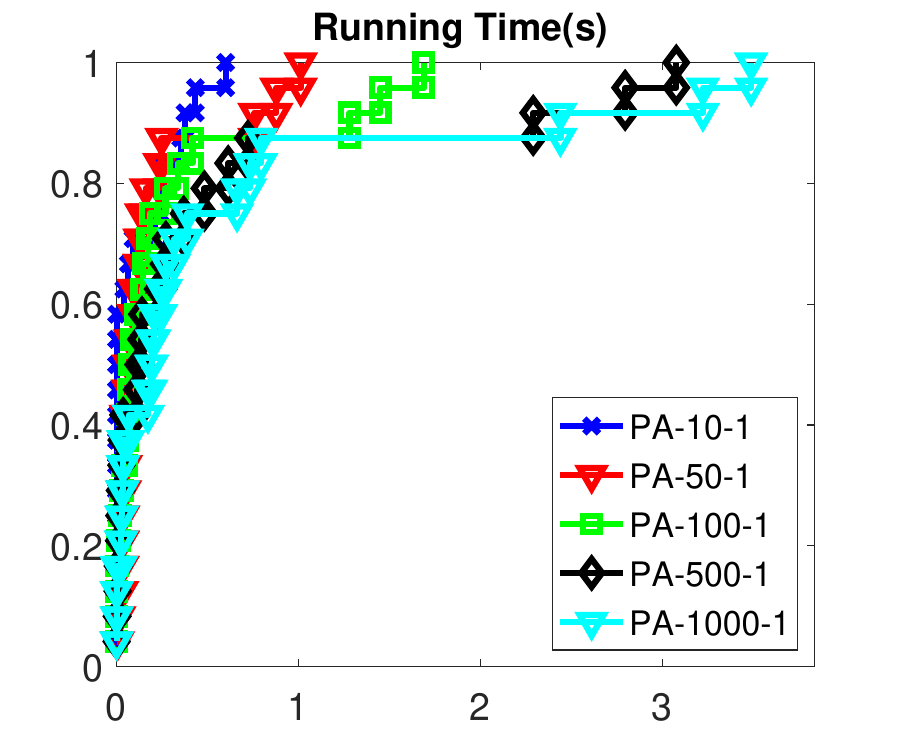}
        \caption{$1$ thread.}
    \end{subfigure}
    \hfill
    \begin{subfigure}[h]{0.5\linewidth}
        \includegraphics[width=\linewidth]{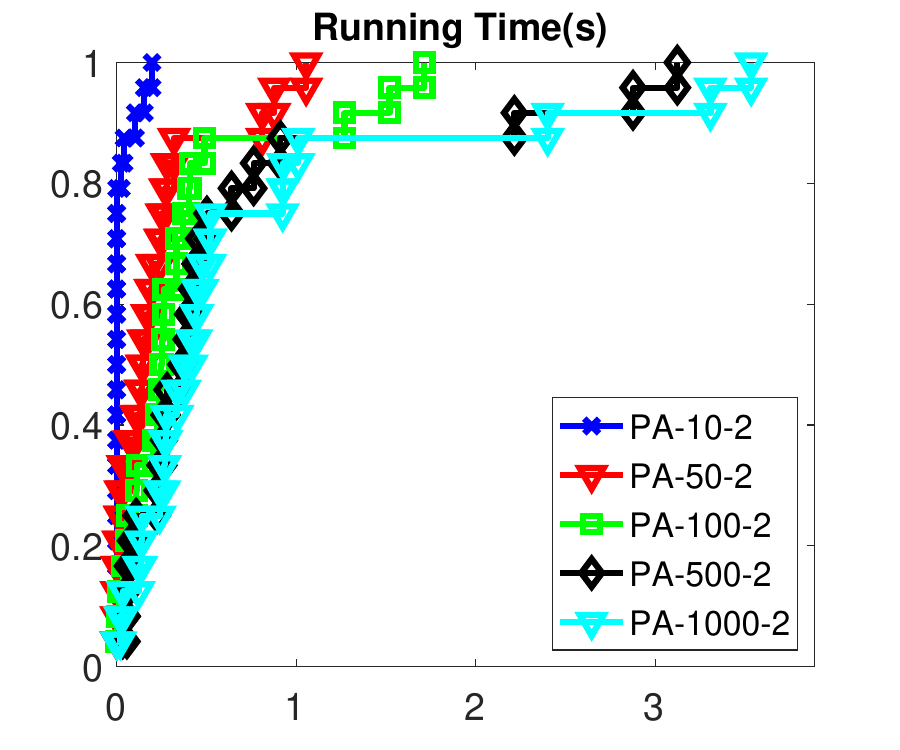}
        \caption{$2$ threads.}
    \end{subfigure}\caption{Performance profile of the average execution time between the variants of the Parallel Active BCDM method
    (\texttt{PA}) with $1$ and $2$ threads for the $24$ $SR$ problems.}
    \label{fig:pa-sr-all1}
\end{figure}

\begin{figure}
    \begin{subfigure}[h]{0.33\linewidth}
        \includegraphics[width=\linewidth]{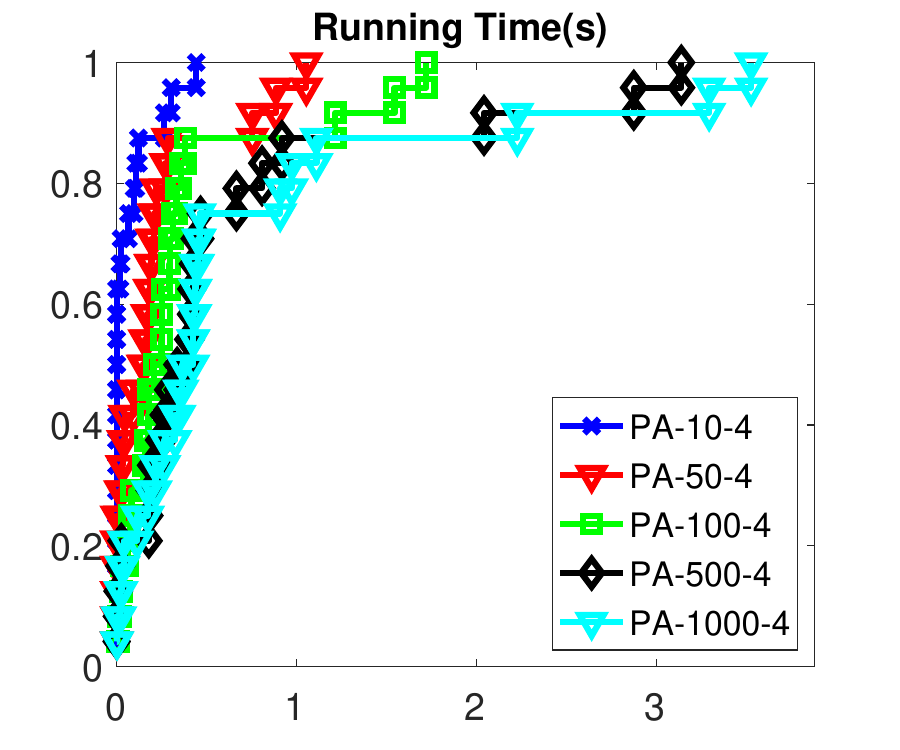}
        \caption{$4$ threads.}
    \end{subfigure}
    \hfill
    \begin{subfigure}[h]{0.33\linewidth}
        \includegraphics[width=\linewidth]{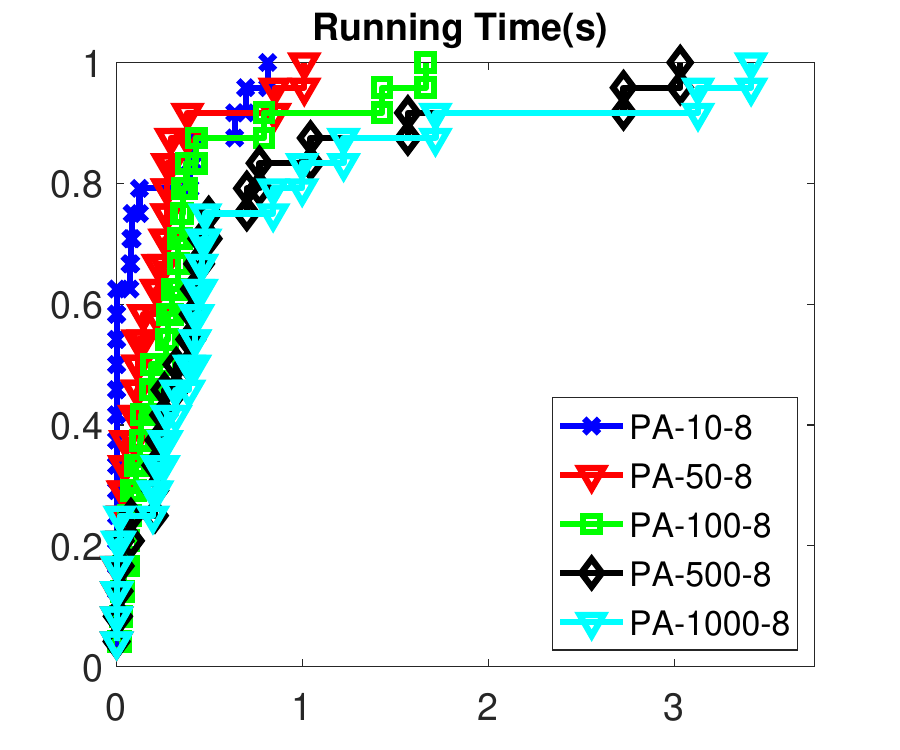}
        \caption{$8$ threads.}
    \end{subfigure}\begin{subfigure}[h]{0.33\linewidth}
        \includegraphics[width=\linewidth]{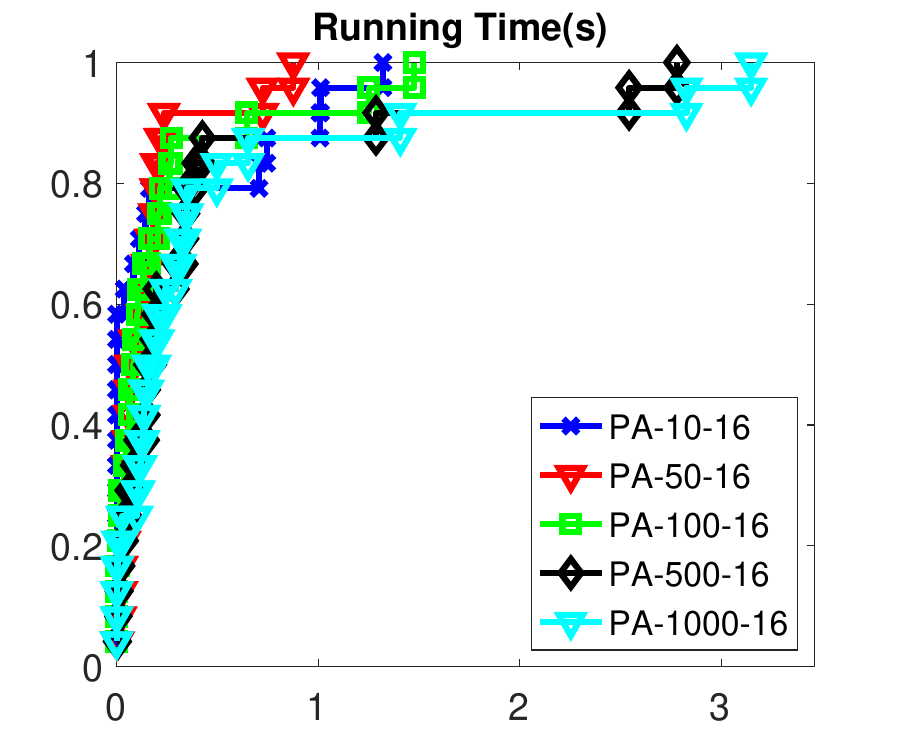}
        \caption{$16$ threads.}
    \end{subfigure}\caption{Performance profile of the average execution time between the variants of the Parallel Active BCDM method
    (\texttt{PA}) with $4$, $8$, and $16$ threads for the $24$ $SR$
    problems.}
    \label{fig:pa-sr-all2}
\end{figure}

\change{ Figure~\ref{fig:rat-sr} presents three boxplots showing
the speedup of the multi-threaded variants when compared to its serial
version. Unlike the previous cases, the multi-threaded variants did not
outperform their serial counterparts in the majority of instances in this test
set. This outcome is attributed to the higher $\tfrac{\omega}{n}$ ratio
observed in this class of problems, which leads to increased values of
$\beta$, thereby penalizing the magnitude of the descent directions when
compared to the serial case. In several instances, the number of iterations
required by the multi-threaded versions increased substantially compared to
the serial versions, resulting in a significant reduction in the effective
speedup. Please refer to the tables in the appendix for further details.}

\change{An analysis of the average speedups reveals that the  \texttt{UBCDM} method achieved speedups of $1.12$, $1.38$, $1.61$, and $1.61$ for $2$, $4$, $8$, and $16$ threads, respectively. The \texttt{PCDM} method reached $1.19$, $1.49$, $2.00$, and $2.50$, while \texttt{PA-10} achieved $1.03$, $1.25$, $1.39$, and $1.44$ for the same thread counts.}

\begin{figure}
    \begin{subfigure}[h]{0.33\linewidth}
        \includegraphics[width=\linewidth]{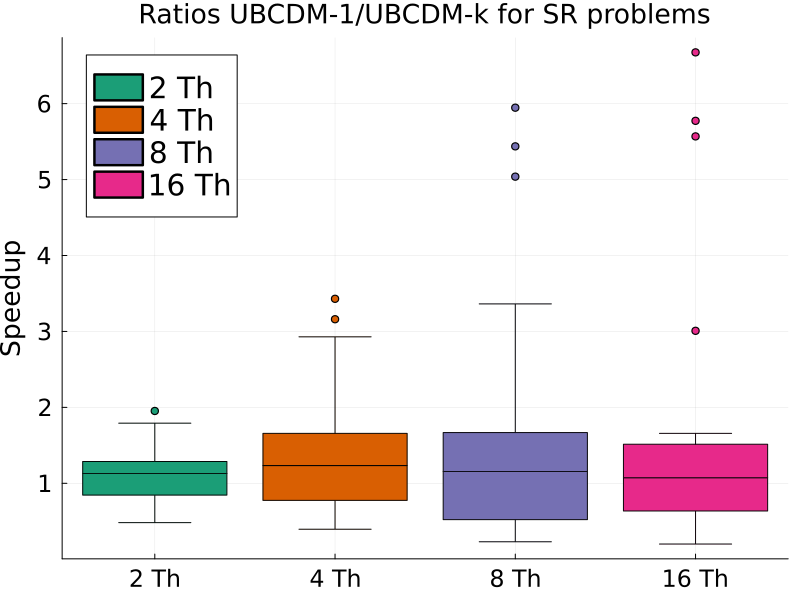}
\end{subfigure}
    \hfill
    \begin{subfigure}[h]{0.33\linewidth}
        \includegraphics[width=\linewidth]{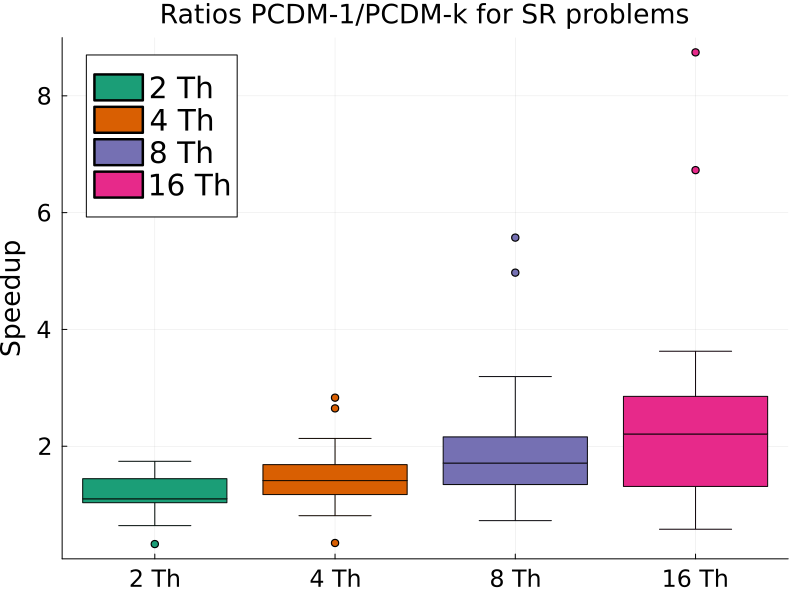}
\end{subfigure}\begin{subfigure}[h]{0.33\linewidth}
        \includegraphics[width=\linewidth]{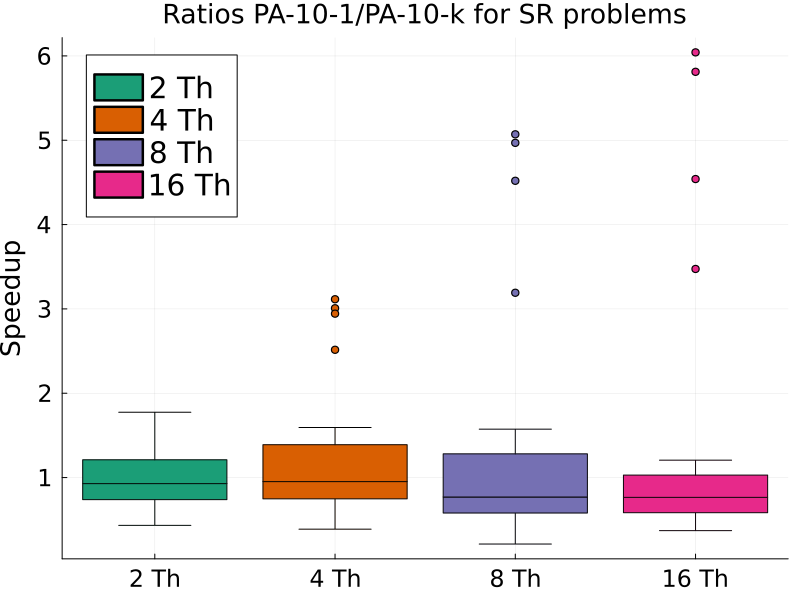}
\end{subfigure}\caption{Boxplots comparing the speedup in running time between
    \texttt{UBCDM-}$1$/\texttt{UBCDM-}$k$ (left),
    \texttt{PCDM-}$1$/\texttt{PCDM-}$k$ (center) and
    \texttt{PA-10-}$1$/\texttt{PA-10-}$k$ (right), with $k\in \{ 2, 4, 8,
    16\}$ for $SR$ problems.}
    \label{fig:rat-sr}
\end{figure}

\change{As with the $SC$ problems, Figure~\ref{fig:rat-sr}
presents boxplots comparing the execution time of the serial implementation of
each method with its corresponding multi-threaded version. In contrast to
previous cases, for the set of SR problems, the multi-threaded implementations
did not outperform their serial counterparts in the majority of instances.
Among the evaluated methods, \texttt{PCDM} demonstrated the most
notable speedup under multi-threaded execution.}

\begin{figure}[h]
    \begin{subfigure}[h]{0.50\linewidth}
        \includegraphics[width=\linewidth]{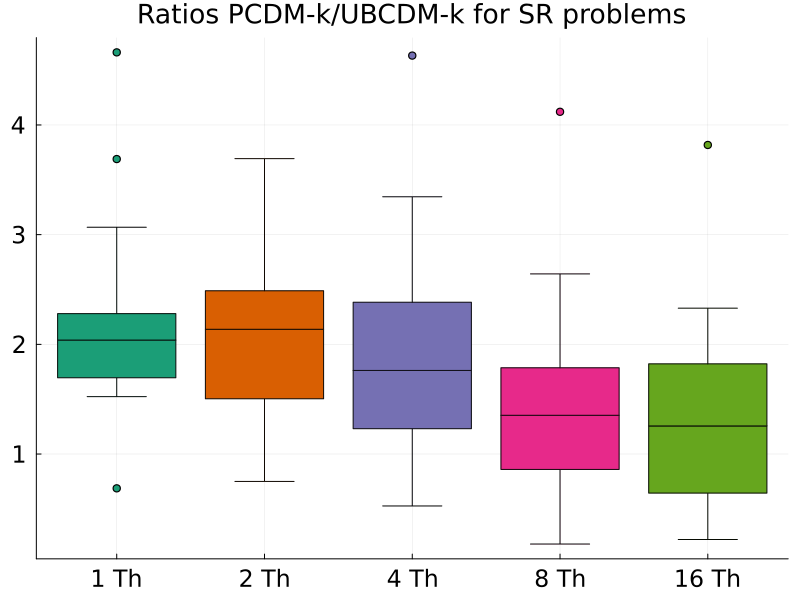}
\end{subfigure}\begin{subfigure}[h]{0.50\linewidth}
    \includegraphics[width=\linewidth]{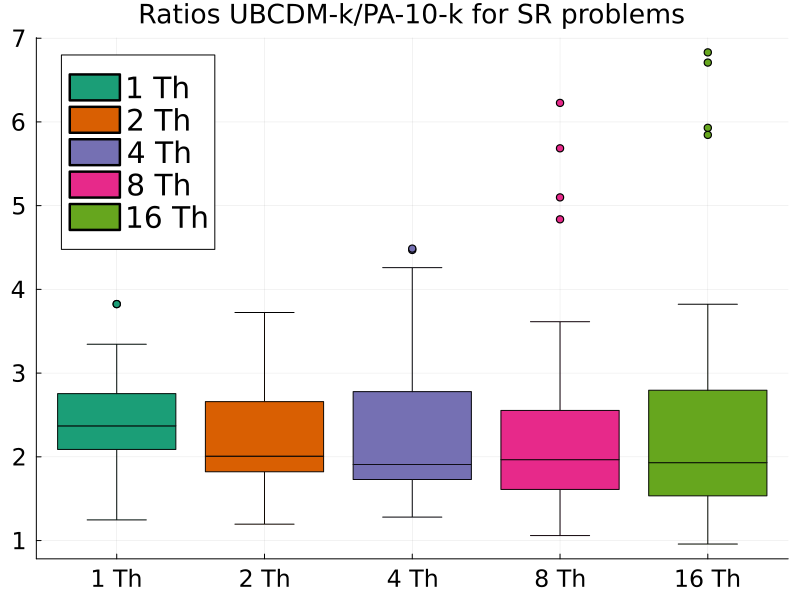}
\end{subfigure}\caption{Boxplot comparing the running time ratio between \texttt{PCDM}
    and \texttt{UBCDM} (left) and between \texttt{UBCDM} and \texttt{PA-10}
    (right) for $SR$ problems, with $1$, $2$, $4$, $8$ and $16$ threads.}
     \label{fig:rat-pbub-sr1}
\end{figure}

\change{To conclude this section of experiments, we present the final
two boxplots in Figure~\ref{fig:rat-pbub-sr1}. It illustrates the
execution time ratios between \texttt{UBCDM} and \texttt{PCDM}, as well as
between \texttt{UBCDM} and \texttt{PA-10}, for the set of problems $SR$. The
results indicate that \texttt{UBCDM} is, on average, approximately $1.5$ times
faster than \texttt{PCDM}, although this advantage slightly decreases as the
number of threads increases to 16. In contrast, the performance gap between
\texttt{PA-10} and \texttt{UBCDM} remains more stable, with \texttt{PA-10}
consistently exhibiting nearly $2.4$ times the demand of \texttt{UBCDM} for
all configurations of threads.}

\change{We recall that Sections~\ref{subs:sc}
and~\ref{subs:sr} of the appendix present tables with full
detailed information about the most relevant methods previously discussed for
problems $SC$ and $SR$.}

 \selectlanguage{english}

\section{Final remarks}
\label{sec:remarks}

In this paper, we \change{introduce} a parallel variation of the Active Block
Coordinate Descent Method from~\cite{LSS19} and \change{show} that its sequence of
objective values converges to the optimal value in expectation, \change{similar} to
the results of~\cite{RT2016} for the uniform case without identification. The
\change{convergence analysis} is accompanied by a high-performance implementation that is tested in
different scenarios based on Lasso problems.

In the synthetic test set presented in~\cite{RT2016}, our implementation
displays a consistent speedup as the number of threads increases. It can also
outperform the implementation from~\cite{RT2016} even in the uniform case,
with further acceleration whenever identification is \change{activated}. In real-world
tests, using the collection from~\cite{LSS19}, the parallel implementation is
still faster than the serial one, but the improvement is more limited than in
the synthetic case.  For problems with sparse matrices \change{that have} a very
unbalanced quantity of elements among the columns, the large amount of
information may burden the computational effort unevenly among the threads. In
case such columns take part \change{in} the problem solution, our identification
strategy should benefit less from the parallelism than \change{a} uniform choice
of blocks. Nevertheless, better speedups were achieved in problems where there
are many columns (features) to be selected. In this scenario, identification
has a favorable effect, \change{decreasing} the computational effort to approximate an
optimal solution.

A future direction of research is to use ideas akin to relative
smoothness~\cite{HR2021} or the smooth approximation framework
in~\cite{CN2024} to relax the differentiability assumptions on the smooth part
of the objective function in a setting that allows the use of identification
of the active constraints. 

\bigskip

\noindent {\bf Acknowledgments.} \change{We are thankful to the
    \change{anonymous} reviewers, whose insightful comments and questions helped us
    improve the presentation of our work. }

\coapstyle{
}
\simplestyle{
    \bibliographystyle{abbrv}
}
\bibliography{LSS_Pbib}

\selectlanguage{english}

\coapstyle{\color{red}}
\section{Appendix}
\label{sec:appendix}

\subsection{Tables - Datasets Ne}\label{subs:ne}

\coapstyle{\arrayrulecolor{red}}
\begin{table}[h!]
    \centering
    \change{
\begin{tabular}{c|c|c|c|c|c} \hline
        \diagbox{Problem}{Threads} & $1T$     & $2T$    & $4T$    & $8T$    & $16T$   \\ \hline
        $Ne1$                      & 1300.837 & 679.493 & 395.705 & 191.876 & 146.266 \\ \hline
        $Ne2$                      & 4.778    & 2.511   & 1.364   & 0.729   & 0.592   \\ \hline
        $Ne3$                      & 83.552   & 44.800  & 22.892  & 11.733  & 9.302   \\ \hline
        $Ne4$                      & 292.040  & 159.256 & 89.240  & 46.335  & 29.576  \\ \hline
        $Ne5$                      & 1.366    & 0.710   & 0.391   & 0.203   & 0.152   \\ \hline
        $Ne6$                      & 29.755   & 14.832  & 7.721   & 4.086   & 2.787   \\ \hline
    \end{tabular}}
    \caption{\change{Running time of the method \texttt{UBCDM} with multiple threads applied to $Ne$ problems.}}
    \label{tab:nest-prob-ub}
\end{table}

\begin{table}[h!]
    \centering
    \change{\begin{tabular}{c|c|c|c|c} \hline
        \diagbox{Problem}{Ratio} & $1T/2T$ & $1T/4T$ & $1T/8T$ & $1T/16T$ \\ \hline
        $Ne1$                    & 1.91442 & 3.28739 & 6.77956 & 8.89365  \\ \hline
        $Ne2$                    & 1.90323 & 3.50227 & 6.55844 & 8.07393  \\ \hline
        $Ne3$                    & 1.86500 & 3.64977 & 7.12090 & 8.98226  \\ \hline
        $Ne4$                    & 1.83377 & 3.27252 & 6.30280 & 9.87430  \\ \hline
        $Ne5$                    & 1.92308 & 3.49023 & 6.71706 & 9.00923  \\ \hline
        $Ne6$                    & 2.00614 & 3.85353 & 7.28232 & 10.6769  \\ \hline
    \end{tabular}}
    \caption{\change{Running time ratio between \texttt{UBCDM-1}/\texttt{UBCDM-k} with \texttt{k} $\in \{ 2, 4, 8, 16\}$, applied to $Ne$ problems.}}
    \label{tab:ra-nest-ub}
\end{table}

\begin{table}[h!]
    \centering
    \change{\begin{tabular}{c|c|c|c|c|c} \hline
        \diagbox{Problem}{Threads} & $1T$               & $2T$               & $4T$               & $8T$               & $16T$              \\ \hline
        $Ne1$                      & $1.20\times10^{9}$ & $1.58\times10^{9}$ & $1.29\times10^{9}$ & $1.32\times10^{9}$ & $1.56\times10^{9}$ \\ \hline
        $Ne2$                      & $6.71\times10^{6}$ & $6.70\times10^{6}$ & $6.74\times10^{6}$ & $6.65\times10^{6}$ & $6.83\times10^{6}$ \\ \hline
        $Ne3$                      & $8.80\times10^{7}$ & $9.04\times10^{7}$ & $8.91\times10^{7}$ & $8.67\times10^{7}$ & $8.71\times10^{7}$ \\ \hline
        $Ne4$                      & $1.40\times10^{9}$ & $1.41\times10^{9}$ & $1.51\times10^{9}$ & $1.52\times10^{9}$ & $1.46\times10^{9}$ \\ \hline
        $Ne5$                      & $8.02\times10^{6}$ & $7.68\times10^{6}$ & $8.04\times10^{6}$ & $7.84\times10^{6}$ & $7.90\times10^{6}$ \\ \hline
        $Ne6$                      & $8.25\times10^{7}$ & $8.41\times10^{7}$ & $8.51\times10^{7}$ & $8.34\times10^{7}$ & $8.44\times10^{7}$ \\ \hline
    \end{tabular}}
    \caption{\change{Average number of iterations of the method \texttt{UBCDM} with multiple threads applied to $Ne$ problems.}}
    \label{tab:nest-it-prob-ub}
\end{table}

\begin{table}[h!]
    \centering
    \change{\begin{tabular}{c|c|c|c|c|c} \hline
        \diagbox{Problem}{Threads} & $1T$   & $2T$   & $4T$   & $8T$    & $16T$   \\ \hline
        $Ne1$                      & 2662.8 & 1910.3 & 917.53 & 404.20  & 246.50  \\ \hline
        $Ne2$                      & 8.8046 & 5.6280 & 3.3843 & 1.7840  & 0.9729  \\ \hline
        $Ne3$                      & 146.49 & 103.62 & 55.365 & 27.569  & 14.400  \\ \hline
        $Ne4$                      & 541.09 & 312.39 & 186.33 & 62.555  & 38.082  \\ \hline
        $Ne5$                      & 2.8220 & 1.8062 & 0.9483 & 0.56064 & 0.26588 \\ \hline
        $Ne6$                      & 50.905 & 31.652 & 17.265 & 9.2756  & 4.9133  \\ \hline
    \end{tabular}}
    \caption{\change{Running time of the method \texttt{PCDM} with multiple threads applied to $Ne$ problems.}}
    \label{tab:nest-prob-pc}
\end{table}

\begin{table}[h!]
    \centering
    \change{\begin{tabular}{c|c|c|c|c} \hline
        \diagbox{Problem}{Ratio} & $1T/2T$ & $1T/4T$ & $1T/8T$ & $1T/16T$ \\ \hline
        $Ne1$                    & 1.39392 & 2.90214 & 6.58783 & 10.8024  \\ \hline
        $Ne2$                    & 1.56443 & 2.60160 & 4.93531 & 9.05013  \\ \hline
        $Ne3$                    & 1.41372 & 2.64590 & 5.31358 & 10.1729  \\ \hline
        $Ne4$                    & 1.73210 & 2.90393 & 8.64983 & 14.2085  \\ \hline
        $Ne5$                    & 1.56240 & 2.97595 & 5.03353 & 10.6138  \\ \hline
        $Ne6$                    & 1.60827 & 2.94845 & 5.48805 & 10.3607  \\ \hline
    \end{tabular}}
    \caption{\change{Running time ratio between \texttt{PCDM-1}/\texttt{PCDM-k} with \texttt{k} $\in \{ 2, 4, 8, 16\}$, applied to $Ne$ problems.}}
    \label{tab:ra-nest-pc}
\end{table}

\begin{table}[h!]
    \centering
    \change{\begin{tabular}{c|c|c|c|c|c} \hline
        \diagbox{Problem}{Threads} & $1T$              & $2T$              & $4T$              & $8T$              & $16T$             \\ \hline
        $Ne1$                      & $2.0\times10^{9}$ & $2.2\times10^{9}$ & $1.9\times10^{9}$ & $1.5\times10^{9}$ & $1.8\times10^{9}$ \\ \hline
        $Ne2$                      & $6.8\times10^{6}$ & $6.4\times10^{6}$ & $6.8\times10^{6}$ & $6.6\times10^{6}$ & $6.7\times10^{6}$ \\ \hline
        $Ne3$                      & $9.0\times10^{7}$ & $8.7\times10^{7}$ & $9.1\times10^{7}$ & $8.7\times10^{7}$ & $8.3\times10^{7}$ \\ \hline
        $Ne4$                      & $1.6\times10^{9}$ & $1.5\times10^{9}$ & $1.6\times10^{9}$ & $1.0\times10^{9}$ & $1.2\times10^{9}$ \\ \hline
        $Ne5$                      & $8.7\times10^{6}$ & $8.7\times10^{6}$ & $8.1\times10^{6}$ & $8.7\times10^{6}$ & $7.2\times10^{6}$ \\ \hline
        $Ne6$                      & $8.0\times10^{7}$ & $8.0\times10^{7}$ & $8.5\times10^{7}$ & $8.8\times10^{7}$ & $8.5\times10^{7}$ \\ \hline
    \end{tabular}}
    \caption{\change{Average number of iterations of the method \texttt{PCDM} with multiple threads applied to $Ne$ problems.}}
    \label{tab:nest-it-prob-pc}
\end{table}

\begin{table}[h!]
    \centering
    \change{\begin{tabular}{c|c|c|c|c|c} \hline
        \diagbox{Problem}{Threads} & $1T$     & $2T$    & $4T$    & $8T$    & $16T$   \\ \hline
        $Ne1$                      & 1021.017 & 610.287 & 323.761 & 184.617 & 140.427 \\ \hline
        $Ne2$                      & 4.390    & 2.283   & 1.203   & 0.661   & 0.543   \\ \hline
        $Ne3$                      & 85.811   & 42.805  & 22.142  & 11.891  & 10.404  \\ \hline
        $Ne4$                      & 227.433  & 118.760 & 66.722  & 32.582  & 21.515  \\ \hline
        $Ne5$                      & 1.182    & 0.699   & 0.344   & 0.187   & 0.147   \\ \hline
        $Ne6$                      & 26.174   & 13.999  & 7.139   & 3.999   & 2.745   \\ \hline
    \end{tabular}}
    \caption{\change{Running time of the method \texttt{PA-500} with multiple threads applied to $Ne$ problems.}}
    \label{tab:nest-prob-pa}
\end{table}

\begin{table}[h!]
    \centering
    \change{\begin{tabular}{c|c|c|c|c} \hline
        \diagbox{Problem}{Ratio} & $1T/2T$ & $1T/4T$ & $1T/8T$ & $1T/16T$ \\ \hline
        $Ne1$                    & 1.67301 & 3.15362 & 5.53045 & 7.27081  \\ \hline
        $Ne2$                    & 1.92254 & 3.64808 & 6.63575 & 8.07794  \\ \hline
        $Ne3$                    & 2.00472 & 3.87544 & 7.21640 & 8.24739  \\ \hline
        $Ne4$                    & 1.91507 & 3.40863 & 6.98027 & 10.5709  \\ \hline
        $Ne5$                    & 1.69169 & 3.43276 & 6.30851 & 8.05658  \\ \hline
        $Ne6$                    & 1.86967 & 3.66637 & 6.54548 & 9.53374  \\ \hline
    \end{tabular}}
    \caption{\change{Running time ratio between \texttt{PA-500-1}/\texttt{PA-500-k} with \texttt{k} $\in \{ 2, 4, 8, 16\}$, applied to $Ne$~problems.}}
    \label{tab:ra-nest-pa}
\end{table}

\begin{table}[h!]
    \centering
    \change{\begin{tabular}{c|c|c|c|c|c} \hline
        \diagbox{Problem}{Threads} & $1T$               & $2T$               & $4T$               & $8T$                & $16T$              \\ \hline
        $Ne1$                      & $1.13\times10^{9}$ & $1.26\times10^{9}$ & $1.23\times10^{9}$ & $1.25\times10^{9}$  & $1.21\times10^{9}$ \\ \hline
        $Ne2$                      & $5.41\times10^{6}$ & $5.33\times10^{6}$ & $5.23\times10^{6}$ & $5.26\times10^{6}$  & $5.35\times10^{6}$ \\ \hline
        $Ne3$                      & $7.52\times10^{7}$ & $7.56\times10^{7}$ & $7.59\times10^{7}$ & $7.37\times10^{7}$  & $7.55\times10^{7}$ \\ \hline
        $Ne4$                      & $1.06\times10^{9}$ & $1.02\times10^{9}$ & $1.10\times10^{9}$ & $1.10\times10^{10}$ & $9.39\times10^{8}$ \\ \hline
        $Ne5$                      & $6.14\times10^{6}$ & $6.63\times10^{6}$ & $6.19\times10^{6}$ & $6.20\times10^{6}$  & $6.40\times10^{6}$ \\ \hline
        $Ne6$                      & $6.81\times10^{7}$ & $6.95\times10^{7}$ & $7.11\times10^{7}$ & $7.15\times10^{7}$  & $6.80\times10^{7}$ \\ \hline
    \end{tabular}}
    \caption{\change{Average number of iterations of the method \texttt{PA-500} with multiple threads applied to $Ne$ problems.}}
    \label{tab:nest_it-prob-pa}
\end{table}

\FloatBarrier
\subsection{Tables - Datasets SC}\label{subs:sc}

\begin{table}[h!]
    \centering
    \change{\begin{tabular}{c|c|c|c|c|c} \hline
        \diagbox{Problem}{Threads} & $1T$     & $2T$     & $4T$     & $8T$     & $16T$    \\ \hline
        SC1                        & 0.145980 & 0.085250 & 0.045240 & 0.024530 & 0.018540 \\ \hline
        SC2                        & 0.224420 & 0.129270 & 0.071410 & 0.039430 & 0.031690 \\ \hline
        SC3                        & 0.282010 & 0.169530 & 0.096300 & 0.056070 & 0.041420 \\ \hline
        SC4                        & 6.302630 & 3.607520 & 2.037880 & 1.313480 & 1.036560 \\ \hline
        SC5                        & 6.317380 & 3.606250 & 2.112590 & 1.308840 & 1.046580 \\ \hline
        SC6                        & 0.091400 & 0.244840 & 0.380470 & 0.904970 & 0.245390 \\ \hline
        SC7                        & 0.118780 & 0.074830 & 0.048980 & 0.029830 & 0.023020 \\
\hline
    \end{tabular}}
    \caption{\change{Running time of the method \texttt{UBCDM} with multiple threads applied to SC* problems.}}
    \label{tab:sc-prob-ubcdm}
\end{table}

\begin{table}[h!]
    \centering
    \change{\begin{tabular}{c|c|c|c|c} \hline
        \diagbox{Problem}{Ratio} & $1T/2T$ & $1T/4T$ & $1T/8T$ & $1T/16T$ \\ \hline
        SC1                      & 1.71238 & 3.22679 & 5.95108 & 7.87379  \\ \hline
        SC2                      & 1.73606 & 3.14270 & 5.69161 & 7.08173  \\ \hline
        SC3                      & 1.66348 & 2.92845 & 5.02961 & 6.80855  \\ \hline
        SC4                      & 1.74708 & 3.09274 & 4.79842 & 6.08033  \\ \hline
        SC5                      & 1.75179 & 2.99035 & 4.82670 & 6.03621  \\ \hline
        SC6                      & 0.37330 & 0.24022 & 0.10099 & 0.37246  \\ \hline
        SC7                      & 1.58733 & 2.42507 & 3.98190 & 5.15986  \\
\hline
    \end{tabular}}
    \caption{\change{Running time ratio between \texttt{UBCDM-1}/\texttt{UBCDM-k} with \texttt{k} $\in \{ 2, 4, 8, 16\}$, applied to SC* problems.}}
    \label{tab:ra-sc-ubcdm}
\end{table}

\begin{table}[h!]
    \centering
    \change{\begin{tabular}{c|c|c|c|c|c} \hline
        \diagbox{Problem}{Threads} & $1T$        & $2T$        & $4T$        & $8T$        & $16T$       \\ \hline
        SC1                        & 956825.60   & 950272.00   & 959447.04   & 956170.24   & 956825.60   \\ \hline
        SC2                        & 870318.08   & 871628.80   & 878182.40   & 877527.04   & 882769.92   \\ \hline
        SC3                        & 1246494.72  & 1247805.44  & 1255014.40  & 1256325.12  & 1275985.92  \\ \hline
        SC4                        & 1032765.44  & 1079214.08  & 1097564.16  & 1187594.24  & 1390592.00  \\ \hline
        SC5                        & 1032765.44  & 1079214.08  & 1111900.16  & 1186447.36  & 1401487.36  \\ \hline
        SC6                        & 23470.08    & 31539.20    & 52838.40    & 96583.68    & 183746.56   \\ \hline
        SC7                        & 657031.26   & 622336.54   & 674378.62   & 652694.42   & 635347.06   \\
\hline
    \end{tabular}}
    \caption{\change{Average number of iterations of the method \texttt{UBCDM} with multiple threads applied to SC* problems.}}
    \label{tab:sc-it-prob-ubcdm}
\end{table}

\begin{table}[h!]
    \centering
    \change{\begin{tabular}{c|c|c|c|c|c} \hline
        \diagbox{Problem}{Threads} & $1T$     & $2T$     & $4T$     & $8T$     & $16T$    \\ \hline
        SC1                        & 0.295980 & 0.191120 & 0.123840 & 0.072854 & 0.048063 \\ \hline
        SC2                        & 0.447870 & 0.290370 & 0.193110 & 0.111220 & 0.068896 \\ \hline
        SC3                        & 0.509380 & 0.324010 & 0.217410 & 0.129210 & 0.079972 \\ \hline
        SC4                        & 12.00500 & 5.951900 & 3.597600 & 2.472100 & 1.790400 \\ \hline
        SC5                        & 11.98600 & 5.906300 & 3.552500 & 2.439900 & 1.764700 \\ \hline
        SC6                        & 0.202090 & 0.142180 & 0.185070 & 0.201200 & 0.276520 \\ \hline
        SC7                        & 0.254240 & 0.252200 & 0.205840 & 0.122730 & 0.105690 \\
\hline
    \end{tabular}}
    \caption{\change{Running time of the method \texttt{PCDM} with multiple threads applied to SC* problems.}}
    \label{tab:sc-prob-pc}
\end{table}

\begin{table}[h!]
    \centering
    \change{\begin{tabular}{c|c|c|c|c} \hline
        \diagbox{Problem}{Ratio} & $1T/2T$ & $1T/4T$ & $1T/8T$ & $1T/16T$ \\ \hline
        SC1                      & 1.54866 & 2.39002 & 4.06265 & 6.15817  \\ \hline
        SC2                      & 1.54241 & 2.31925 & 4.02688 & 6.50067  \\ \hline
        SC3                      & 1.57211 & 2.34295 & 3.94226 & 6.36948  \\ \hline
        SC4                      & 2.01700 & 3.33695 & 4.85620 & 6.70521  \\ \hline
        SC5                      & 2.02936 & 3.37396 & 4.91250 & 6.79209  \\ \hline
        SC6                      & 1.42137 & 1.09197 & 1.00442 & 0.73083  \\ \hline
        SC7                      & 1.00809 & 1.23513 & 2.07154 & 2.40553  \\
\hline
    \end{tabular}}
    \caption{\change{Running time ratio between \texttt{PCDM-1}/\texttt{PCDM-k} with \texttt{k} $\in \{ 2, 4, 8, 16\}$, applied to NE* problems.}}
    \label{tab:ra-sc-pc}
\end{table}

\begin{table}[h!]
    \centering
    \change{\begin{tabular}{c|c|c|c|c|c} \hline
        \diagbox{Problem}{Threads} & $1T$       & $2T$       & $4T$       & $8T$        & $16T$       \\ \hline
        SC1                        & 917504.0   & 917504.0   & 917504.0   & 851968.0    & 851968.0    \\ \hline
        SC2                        & 851968.0   & 851968.0   & 851968.0   & 786432.0    & 786432.0    \\ \hline
        SC3                        & 1179648.0  & 1179648.0  & 1179648.0  & 1179648.0   & 1179648.0   \\ \hline
        SC4                        & 1089536.0  & 860160.0   & 916930.56  & 1261568.0   & 1662976.0   \\ \hline
        SC5                        & 1089536.0  & 860160.0   & 916930.56  & 1261568.0   & 1664122.88  \\ \hline
        SC6                        & 20480.0    & 20480.0    & 45056.0    & 90112.0     & 184320.0    \\ \hline
        SC7                        & 433684.0   & 650526.0   & 650526.0   & 216842.0    & 216842.0    \\
\hline
    \end{tabular}}
    \caption{\change{Average number of iterations of the method \texttt{PCDM} with multiple threads applied to SC* problems.}}
    \label{tab:sc-it-prob-pc}
\end{table}

\begin{table}[h!]
    \centering
    \change{\begin{tabular}{c|c|c|c|c|c} \hline
        \diagbox{Problem}{Threads} & $1T$     & $2T$     & $4T$     & $8T$     & $16T$    \\ \hline
        SC1                        & 0.047730 & 0.027920 & 0.015170 & 0.008850 & 0.007564 \\ \hline
        SC2                        & 0.076200 & 0.044280 & 0.024760 & 0.014420 & 0.013210 \\ \hline
        SC3                        & 0.098590 & 0.059140 & 0.033900 & 0.019760 & 0.015040 \\ \hline
        SC4                        & 1.383670 & 0.827150 & 0.540110 & 0.419150 & 0.429240 \\ \hline
        SC5                        & 1.385970 & 0.827140 & 0.525250 & 0.423070 & 0.449980 \\ \hline
        SC6                        & 0.037660 & 0.071650 & 0.065810 & 0.115420 & 0.032920 \\ \hline
        SC7                        & 0.086440 & 0.061200 & 0.040360 & 0.027550 & 0.022440 \\
\hline
    \end{tabular}}
    \caption{\change{Running time of the method \texttt{PA-10} with multiple threads applied to SC* problems.}}
    \label{tab:sc-prob-pa}
\end{table}

\begin{table}[h!]
    \centering
    \change{\begin{tabular}{c|c|c|c|c} \hline
        \diagbox{Problem}{Ratio} & $2T/1T$ & $4T/1T$ & $8T/1T$ & $16T/1T$ \\ \hline
        SC1                      & 1.70953 & 3.14634 & 5.39322 & 6.31015  \\ \hline
        SC2                      & 1.72087 & 3.07754 & 5.28433 & 5.76836  \\ \hline
        SC3                      & 1.66706 & 2.90826 & 4.98937 & 6.55519  \\ \hline
        SC4                      & 1.67282 & 2.56183 & 3.30113 & 3.22353  \\ \hline
        SC5                      & 1.67562 & 2.63869 & 3.27598 & 3.08007  \\ \hline
        SC6                      & 0.52561 & 0.57225 & 0.32628 & 1.14399  \\ \hline
        SC7                      & 1.41242 & 2.14172 & 3.13757 & 3.85205  \\
\hline
    \end{tabular}}
    \caption{\change{Running time ratio between \texttt{PA-10-1}/\texttt{PA-10-k} with \texttt{k} $\in \{ 2, 4, 8, 16\}$, applied to SC* problems.}}
    \label{tab:ra-sc-pa}
\end{table}

\begin{table}[h!]
    \centering
    \change{\begin{tabular}{c|c|c|c|c|c} \hline
        \diagbox{Problem}{Threads} & $1T$       & $2T$       & $4T$       & $8T$       & $16T$      \\ \hline
        SC1                        & 275251.20  & 273285.12  & 275251.20  & 284182.66  & 308955.43  \\ \hline
        SC2                        & 265420.80  & 265420.80  & 270663.68  & 275251.20  & 312606.72  \\ \hline
        SC3                        & 393216.00  & 393216.00  & 393216.00  & 393216.00  & 401080.32  \\ \hline
        SC4                        & 202424.32  & 217333.76  & 248872.96  & 306216.96  & 415744.00  \\ \hline
        SC5                        & 202424.32  & 220200.96  & 246005.76  & 309084.16  & 428359.68  \\ \hline
        SC6                        & 8560.64    & 8765.44    & 11919.36   & 16506.88   & 25600.00   \\ \hline
        SC7                        & 420673.48  & 427178.74  & 427178.74  & 425010.32  & 429347.16  \\
\hline
    \end{tabular}}
    \caption{\change{Average number of iterations of the method \texttt{PA-10} with multiple threads applied to SC* problems.}}
    \label{tab:sc-prob-pa2}
\end{table}

\FloatBarrier
\subsection{Tables - Datasets SR}\label{subs:sr}

\begin{table}[h!]
    \centering
    \change{\begin{tabular}{c|c|c|c|c|c} \hline
        \diagbox{Problem}{Threads} & $1T$      & $2T$    & $4T$    & $8T$    & $16T$   \\ \hline
        SR1                        & 0.060170  & 0.04674 & 0.03554 & 0.03524 & 0.03864 \\ \hline
        SR2                        & 0.040430  & 0.03136 & 0.02424 & 0.02313 & 0.02439 \\ \hline
        SR3                        & 0.072170  & 0.05771 & 0.04363 & 0.04360 & 0.0481  \\ \hline
        SR4                        & 0.500270  & 0.49184 & 0.43086 & 0.42877 & 0.36076 \\ \hline
        SR5                        & 0.006308  & 0.01027 & 0.01578 & 0.02741 & 0.00986 \\ \hline
        SR6                        & 2.301980  & 2.36673 & 2.11020 & 2.01118 & 1.95690 \\ \hline
        SR7                        & 0.076220  & 0.06340 & 0.05882 & 0.07350 & 0.06652 \\
\hline
    \end{tabular}}
    \caption{\change{Running time of the method \texttt{UBCDM} with multiple threads applied to SR* problems.}}
    \label{tab:sr-prob-ubcdm}
\end{table}

\begin{table}[h!]
    \centering
    \change{\begin{tabular}{c|c|c|c|c} \hline
        \diagbox{Problem}{Ratio} & $1T/2T$ & $1T/4T$ & $1T/8T$ & $1T/16T$ \\ \hline
        SR1                      & 1.28733 & 1.69302 & 1.70743 & 1.55719  \\ \hline
        SR2                      & 1.28922 & 1.66790 & 1.74795 & 1.65765  \\ \hline
        SR3                      & 1.25056 & 1.65414 & 1.65528 & 1.50042  \\ \hline
        SR4                      & 1.01714 & 1.16110 & 1.16676 & 1.38671  \\ \hline
        SR5                      & 0.61421 & 0.39974 & 0.23013 & 0.63956  \\ \hline
        SR6                      & 0.97264 & 1.09088 & 1.14459 & 1.17634  \\ \hline
        SR7                      & 1.20221 & 1.29582 & 1.03701 & 1.14582  \\
\hline
    \end{tabular}}
    \caption{\change{Running time ratio between \texttt{UBCDM-1}/\texttt{UBCDM-k} with \texttt{k} $\in \{ 2, 4, 8, 16\}$, applied to SR* problems.}}
    \label{tab:ra-sr-ubcdm}
\end{table}

\begin{table}[h!]
    \centering
    \change{\begin{tabular}{c|c|c|c|c|c} \hline
        \diagbox{Problem}{Threads} & $1T$      & $2T$      & $4T$      & $8T$      & $16T$     \\ \hline
        SR1                        & 3236.13   & 3563.31   & 4189.38   & 5496.87   & 8215.17   \\ \hline
        SR2                        & 2264.43   & 2483.37   & 2959.38   & 3724.44   & 5259.48   \\ \hline
        SR3                        & 4409.55   & 4933.53   & 5747.79   & 7575.57   & 11319.69  \\ \hline
        SR4                        & 4089.96   & 4959.36   & 7284.06   & 10944.36  & 14779.80  \\ \hline
        SR5                        & 2531.32   & 3043.80   & 4071.60   & 6001.20   & 9382.94   \\ \hline
        SR6                        & 28794.72  & 38173.08  & 61296.33  & 96988.59  & 158033.97 \\ \hline
        SR7                        & 9765.28   & 11109.28  & 14675.36  & 21209.44  & 35086.24  \\
\hline
    \end{tabular}}
    \caption{\change{Average number of iterations of the method \texttt{UBCDM} with multiple threads applied to SR* problems.}}
    \label{tab:sr-it-prob-ubcdm}
\end{table}

\begin{table}[h!]
    \centering
    \change{\begin{tabular}{c|c|c|c|c|c} \hline
        \diagbox{Problem}{Threads} & $1T$    & $2T$    & $4T$    & $8T$     & $16T$   \\ \hline
        SR1                        & 0.12204 & 0.12296 & 0.11422 & 0.084881 & 0.07406 \\ \hline
        SR2                        & 0.09174 & 0.08349 & 0.08109 & 0.05328  & 0.04713 \\ \hline
        SR3                        & 0.15507 & 0.14088 & 0.13083 & 0.11519  & 0.10124 \\ \hline
        SR4                        & 0.34371 & 1.0583  & 1.0015  & 0.42095  & 0.46745 \\ \hline
        SR5                        & 0.01458 & 0.01119 & 0.01128 & 0.01088  & 0.01231 \\ \hline
        SR6                        & 4.2099  & 2.5719  & 2.7691  & 3.1289   & 2.4748  \\ \hline
        SR7                        & 0.18007 & 0.15135 & 0.12963 & 0.11744  & 0.15223 \\
\hline
    \end{tabular}}
    \caption{\change{Running time of the method \texttt{PCDM} with multiple threads applied to SR* problems.}}
    \label{tab:sr-prob-pc}
\end{table}

\begin{table}[h!]
    \centering
    \change{\begin{tabular}{c|c|c|c|c} \hline
        \diagbox{Problem}{Ratio} & $1T/2T$ & $1T/4T$ & $1T/8T$ & $1T/16T$ \\ \hline
        SR1                      & 0.99251 & 1.06846 & 1.43778 & 1.64785  \\ \hline
        SR2                      & 1.09871 & 1.13125 & 1.72159 & 1.94616  \\ \hline
        SR3                      & 1.10072 & 1.18528 & 1.34621 & 1.53171  \\ \hline
        SR4                      & 0.32477 & 0.34319 & 0.81651 & 0.73528  \\ \hline
        SR5                      & 1.30352 & 1.29266 & 1.33979 & 1.18475  \\ \hline
        SR6                      & 1.63688 & 1.52031 & 1.34549 & 1.70111  \\ \hline
        SR7                      & 1.18976 & 1.38911 & 1.53329 & 1.18288  \\
\hline
    \end{tabular}}
    \caption{\change{Running time ratio between \texttt{PCDM-1}/\texttt{PCDM-k} with \texttt{k} $\in \{ 2, 4, 8, 16\}$, applied to SR* problems.}}
    \label{tab:ra-sr-pc}
\end{table}

\begin{table}[h!]
    \centering
    \change{\begin{tabular}{c|c|c|c|c|c} \hline
        \diagbox{Problem}{Threads} & $1T$     & $2T$     & $4T$      & $8T$      & $16T$     \\ \hline
        SR1                        & 2952.0   & 5170.92  & 6258.24   & 7165.98   & 7867.08   \\ \hline
        SR2                        & 2214.0   & 3444.0   & 4466.13   & 4250.88   & 5033.16   \\ \hline
        SR3                        & 4305.0   & 6642.0   & 8141.37   & 10830.15  & 12315.99  \\ \hline
        SR4                        & 1260.0   & 8195.04  & 11085.48  & 8328.6    & 14467.32  \\ \hline
        SR5                        & 2160.0   & 2280.6   & 3434.4    & 5601.6    & 7803.0    \\ \hline
        SR6                        & 30831.0  & 31540.83 & 52591.95  & 106338.27 & 150892.65 \\ \hline
        SR7                        & 10976.0  & 15444.8  & 20916.0   & 28716.8   & 40237.12  \\
\hline
    \end{tabular}}
    \caption{\change{Average number of iterations of the method \texttt{PCDM} with multiple threads applied to SR* problems.}}
    \label{tab:sr-it-prob-pc}
\end{table}

\begin{table}[h!]
    \centering
    \change{\begin{tabular}{c|c|c|c|c|c} \hline
        \diagbox{Problem}{Threads} & $1T$    & $2T$    & $4T$    & $8T$    & $16T$   \\ \hline
        SR1                        & 0.02429 & 0.02678 & 0.02431 & 0.02913 & 0.0336  \\ \hline
        SR2                        & 0.01820 & 0.02011 & 0.01835 & 0.02183 & 0.02546 \\ \hline
        SR3                        & 0.02629 & 0.02958 & 0.02877 & 0.0334  & 0.04079 \\ \hline
        SR4                        & 0.19722 & 0.29489 & 0.26163 & 0.27652 & 0.25714 \\ \hline
        SR5                        & 0.00506 & 0.00859 & 0.01233 & 0.02394 & 0.00950 \\ \hline
        SR6                        & 0.60193 & 0.90691 & 0.75186 & 0.70294 & 0.65568 \\ \hline
        SR7                        & 0.02279 & 0.03072 & 0.03349 & 0.04509 & 0.04320 \\
\hline
    \end{tabular}}
    \caption{\change{Running time of the method \texttt{PA-10} with multiple threads applied to SR* problems.}}
    \label{tab:sr-prob-pa}
\end{table}

\begin{table}[h!]
    \centering
    \change{\begin{tabular}{c|c|c|c|c} \hline
        \diagbox{Problem}{Ratio} & $2T/1T$ & $4T/1T$ & $8T/1T$ & $16T/1T$ \\ \hline
        SR1                      & 0.90702 & 0.99917 & 0.83384 & 0.72291  \\ \hline
        SR2                      & 0.90502 & 0.99182 & 0.83371 & 0.71484  \\ \hline
        SR3                      & 0.88877 & 0.91379 & 0.78712 & 0.64452  \\ \hline
        SR4                      & 0.66879 & 0.75381 & 0.71322 & 0.76697  \\ \hline
        SR5                      & 0.58917 & 0.41046 & 0.21140 & 0.53273  \\ \hline
        SR6                      & 0.66371 & 0.80058 & 0.85630 & 0.91802  \\ \hline
        SR7                      & 0.74186 & 0.68050 & 0.50543 & 0.52754  \\
\hline
    \end{tabular}}
    \caption{\change{Running time ratio between \texttt{PA-10-1}/\texttt{PA-10-k} with \texttt{k} $\in \{ 2, 4, 8, 16\}$, applied to SR* problems.}}
    \label{tab:ra-sr-pa}
\end{table}

\begin{table}[h!]
    \centering
    \change{\begin{tabular}{c|c|c|c|c|c} \hline
        \diagbox{Problem}{Threads} & $1T$      & $2T$      & $4T$      & $8T$      & $16T$     \\ \hline
        SR1                        & 685.11    & 907.74    & 1291.50   & 2102.07   & 3713.37   \\ \hline
        SR2                        & 562.11    & 709.71    & 1001.22   & 1627.29   & 2858.52   \\ \hline
        SR3                        & 880.68    & 1132.83   & 1725.69   & 2740.44   & 5109.42   \\ \hline
        SR4                        & 1120.14   & 1679.58   & 2613.24   & 4393.62   & 7247.52   \\ \hline
        SR5                        & 1809.09   & 2384.62   & 3450.60   & 5583.60   & 8894.72   \\ \hline
        SR6                        & 5384.67   & 7614.54   & 11464.83  & 19043.52  & 31971.03  \\ \hline
        SR7                        & 1682.24   & 2261.28   & 3690.40   & 6588.96   & 11966.08  \\
\hline
    \end{tabular}}
    \caption{\change{Average number of iterations of the method \texttt{PA-10} with multiple threads applied to SR* problems.}}
    \label{tab:sr-it-prob-pa}
\end{table}
\coapstyle{\color{black}} \end{document}